\documentclass[11pt,a4paper, envcountsame]{amsart}
\usepackage[usenames,dvipsnames]{color}

\usepackage[hyperindex,breaklinks,colorlinks=true,linkcolor=black,anchorcolor=black,citecolor=black,filecolor=black,menucolor=black,runcolor=black,urlcolor=blue,pagebackref]{hyperref}
 \usepackage{amsmath, amssymb, xspace}
 \usepackage{graphicx}

 \usepackage[all]{xy}

\usepackage{tikz}
\usetikzlibrary{arrows,snakes,positioning,backgrounds,shadows}

\newtheorem{theorem}{Theorem}[section]

\newtheorem{thm}[theorem]{Theorem}

\newtheorem{fact}[theorem]{Fact}

\newtheorem{proposition}[theorem]{Proposition}

\newtheorem{remark}[theorem]{Remark}

\newtheorem{lemma}[theorem]{Lemma}

\newtheorem{cor}[theorem]{Corollary}

\newtheorem{question}[theorem]{Question}

\theoremstyle{definition}
\newtheorem{definition}[theorem]{Definition}

\newcommand{\DII}{\Delta^0_2}

\newcommand{\QQ}{{\mathbb{Q}}}

\newcommand{\sub}{\subseteq}
\newcommand{\sN}[1]{_{#1\in \omega}}
\newcommand{\uhr}[1]{\! \upharpoonright_{#1}}

\newcommand{\SI}[1]{\Sigma^0_{#1}}
\newcommand{\PI}[1]{\Pi^0_{#1}}

\newcommand{\bi}{\begin{itemize}}
\newcommand{\ei}{\end{itemize}}
\newcommand{\bc}{\begin{center}}
\newcommand{\ec}{\end{center}}

\newcommand{\ES}{\emptyset}
\newcommand{\estring}{\emptyset}

\newcommand{\ex}{\exists}
\newcommand{\fa}{\forall}

\newcommand{\la}{\langle}
\newcommand{\ra}{\rangle}

\newcommand{\n}{\noindent}

\newcommand{\vsps}{\vspace{3pt}}

\newcommand{\lwtt}{\le_{\mathrm{wtt}}}

\newcommand{\aaa}{\alpha}

\newcommand{\w}{\omega}

\def\bfz{\boldsymbol{0}}

\newcommand\+[1]{\mathcal{#1}}

\newcommand{\lra}{\leftrightarrow}

\newcommand{\PP}{\ensuremath{\mathrm{P}}}
\newcommand{\DTIME}{\ensuremath{\mathrm{D{\scriptstyle TIME}}}}

\newcommand{\set}[2]{\ensuremath{ \{ #1 : #2 \} }}
\newcommand{\Ece}[1]{\ensuremath{E^{ce}_{#1}}}

 \renewcommand{\deg}{\mbox{deg}}

\newenvironment{pf}{\begin{trivlist}\item[\hskip\labelsep
{\it Proof.}]}{\end{trivlist}}

\begin{document}

\title[Complexity of equivalence relations and preorders]{Complexity of  equivalence relations and preorders from computability theory}

\author[E. Ianovski]{Egor Ianovski}
\address{Department of Computer Science,  University of Oxford, Wolfson Building, Parks Road, Oxford OX1 3QD  UK.}
\email{egor.ianovski@cs.ox.ac.uk }

\author[R. Miller]{Russell Miller}
\address{Department of Mathematics, Queens College, 65-30 Kissena Blvd.,
Flushing, NY 11367  USA; \& Ph.D.\ Programs in Mathematics \& Computer Science,
CUNY Graduate Center, 365 Fifth Avenue, New York, NY 10016  USA.}
\email{ Russell.Miller@qc.cuny.edu}

\author[K.M. Ng]{Keng Meng Ng}
\address{Division of Mathematical Sciences, School of Physical \& Mathematical Sciences,
Nanyang Technological University, 21 Nanyang Link, Singapore}
\email{kmng@ntu.edu.sg}

\author[A. Nies]{Andr\'e Nies}
\address{Department of Computer Science, University of Auckland, Private
Bag 92019, Auckland, New Zealand.}
\email{andre@cs.auckland.ac.nz}

%\author[F. Stephan]{Frank Stephan}
%\address{Department of Computer Science, University of Auckland, Private
%Bag 92019, Auckland, New Zealand.}
%\email{@cs.auckland.ac.nz}%

\thanks{Miller was
partially supported by Grant \# DMS -- 1001306 from
the National Science Foundation, and
by the European Science Foundation.
Nies  is supported by the Marsden Fund of New Zealand. Miller and Nies were also supported by the  Isaac Newton Institute as  visiting fellows.}

%, and by grant \#
%64229-00 42 from the PSC-CUNY Research Award Program TOO MUCH!

\makeatother


\begin{abstract}
	We study the relative complexity of equivalence relations and preorders from computability theory and  complexity theory. Given  binary relations $R, S$, a  componentwise reducibility is defined by   \bc $
R\le S \iff  \ex f   \, \forall x, y \,  [xRy \lra  f(x) Sf(y)].
$\ec
Here $f$ is taken from a suitable class of effective functions.
For us  the relations will be on natural numbers, and $f$ must be computable.   We show that there is a $\PI 1$-complete equivalence relation, but no $\PI k$-complete for $k \ge 2$.
  We show that  $\SI k$ preorders arising naturally in the above-mentioned areas are $\SI k$-complete.  This includes  polynomial time $m$-reducibility on exponential time sets, which is $\SI 2$,   almost inclusion on r.e.\ sets, which is $\SI 3$, and Turing reducibility on r.e.\ sets, which is~$\SI 4$.
\end{abstract}

\maketitle

\section{Introduction}
Mathematicians devote much of their energy to the classification
of mathematical structures.  In this endeavor, the tool most commonly
used is the notion of a \emph{reduction}:  a map reducing a seemingly
more complicated question to a simpler question.  As a simple example,
one classifies the countable rational vector spaces (under the
equivalence relation of isomorphism) using the notion of dimension, by proving
that two such vector spaces are isomorphic if and only if they have the same
dimension over $\QQ$.  More generally, we have the following definition.
\begin{definition}
\label{defn:reduction}
If $E$ and $F$ are equivalence relations on the
domains $D_E$ and $D_F$, a \emph{reduction} from $E$ to $F$
is a function $f:D_E\to D_F$ with the property that, for all $x,y\in D_E$,
$$ x~E~y\iff f(x)~F~f(y).$$
\end{definition}
\n In our example, $E$ is the isomorphism relation on the class $D_E$
of countable rational vector spaces (with domain $\omega$, say, to avoid
set-theoretical issues), and $F$ is the equality relation on the set
$\{ 0,1,2,\ldots,\omega\}$ of possible dimensions of these spaces.

A reduction $f$ of a complicated equivalence relation $E$ to a simpler
equivalence relation $F$ is entirely possible if one allows $f$ to be
sufficiently complex.  (For instance, every equivalence reduction $E$ with exactly
$n$ classes can be reduced to the equality relation on $\{ 1,2,\ldots,n\}$,
but this reduction has the same complexity as $E$.)
Normally, however, the goal is for $f$ to be a readily
understandable function, so that we can actually learn something from the reduction.
In our example, one should ask how hard it is to
compute the dimension of a rational vector space.  It is natural to restrict
the question to computable vector spaces over $\mathbb Q$  (i.e.\ those where the vector addition
is given as a Turing-computable function on the domain $\omega$
of the space).  Yet even when its domain $D_E$ such 
  vector spaces, computing the function which maps each one to its dimension
requires a $\bfz'''$-oracle, hence is not as simple as one might have hoped.
(The reasons why $\bfz'''$ is required can be gleaned from \cite{Calvert:04} 
or \cite{CalvertHarizanovKnightMiller:06}.)

\subsection*{Effective reductions} Reductions are normally ranked by the ease of computing them.
In the context of Borel theory, for instance, a large body of research is devoted
to the study of \emph{Borel reductions} (the standard book  reference is~\cite{Gao:09}).  Here the domains $D_E$ and $D_F$
are the set $2^\omega$  or some other standard Borel space, and a Borel reduction $f$ is a
reduction (from $E$ to $F$, these being equivalence relations on $2^\omega$)
which, viewed as a function from $2^\omega$ to $2^\omega$,
is Borel.  If such a reduction exists, one says that $E$ is \emph{Borel reducible to} $F$,
and writes $E\leq_B F$.
A stronger possible requirement is that $f$ be continuous,
in which case we have (of course) a \emph{continuous reduction}.
In case the reduction is given by a Turing functional from reals to reals,
it is a \emph{(type-2) computable reduction}.

A further body of research is devoted to the study of the same question
for equivalence relations $E$ and $F$ on $\omega$, and reductions
$f:\omega\to\omega$ between them which are computable.
If such a reduction from $E$ to $F$ exists, we say that $E$ is
\emph{computably reducible to} $F$, and write $E\leq_c F$,
or often just $E\leq F$.
These reductions will be the focus of this
paper. Computable reducibility on equivalence relations was perhaps first studied by Ershov \cite{Ershov:77} in a category theoretic setting. %For  recent work in this direction see~\cite{Fokina.Friedman.Harizanov.Knight.McCoy.Montalban:10,}.

% In certain contexts
%we will also consider $\bfd$-computable reductions for various oracle
%sets $D$ of Turing degree $\bfd$.  Here, $\bfd$-computable reducibility
%is denoted by $\leq_{\bfd}$.  Finally, when we wish to consider
%complexity instead of computability, we will speak of \emph{polynomial-time
%reductions} and the like.

The main purpose of this paper is to investigate the complexity of equivalence
relations under these reducibilities.
%, and to determine when complete equivalence relations exist and which naturally occurring equivalence relations are complete under $\leq_c$ at various levels of the arithmetical hierarchy.
In certain cases we will generalize from
equivalence relations to preorders on $\omega$. We restrict most of our discussion to relatively
low levels of the hierarchy, usually to $\Pi^0_n$ and $\Sigma^0_n$ with $n\leq 4$.
One can focus more closely on very low levels:  Such articles as
\cite{Andrews.Lempp.etal:13,Gao.Gerdes:01,Bernardi.Sorbi:83},
for instance, have dealt exclusively with $\Sigma^0_1$ equivalence relations. The work \cite{Fokina.Friedman.ea:12} considers arithmetical equivalence relations, in particular $\Sigma^0_3$ ones.
At the other extreme, the papers
\cite{Fokina.Friedman.Harizanov.Knight.McCoy.Montalban:10,Fokina.Friedman:11}
consider equivalence relations outside the arithmetical or  even hyperarithmetical, obtaining certain results about $\Sigma^1_1$-completeness
of various equivalence relations on $\omega$.  The work
\cite{Coskey.Hamkins.Miller:12} is focused on much the same
levels of the arithmetical hierarchy as our work here, but through
the prism of Borel equivalence relations on $2^\omega$,
by restricting such relations to the class of computably enumerable
subsets of $\omega$.  A few of their results will be extended here.
Finally, the complexity of equivalence relations under this kind of reducibility has
recently been studied by a number of authors from various disciplines. Within
complexity theory it seems to have been first introduced in
\cite{Fortnow.Grochow:11}, and later studied in \cite{Buss.Chen.Flum.Friedman.Muller:11}.

\subsection*{Completeness} Equivalence relations on $\omega$ are subsets of $\omega^2$, and
so those with arithmetical definitions can be
ranked in the usual hierarchy as $\Sigma^0_n$ and/or $\Pi^0_n$
for various $n\in\omega$.  The question of \emph{completeness}
then naturally arises.  A  $\Sigma^0_n$ equivalence relation $F$ is \emph{complete}
(among  equivalence relations under $\leq_c$)
if   for every $\Sigma^0_n$ equivalence
relation $E$, we have $E\leq_c F$.  $\Pi^0_n$-completeness is defined
similarly.  Much of this paper is devoted to the study of complete     equivalence
relations at various levels of the arithmetical hierarchy, and we will often
speak of their  $\Sigma^0_n$-completeness or $\Pi^0_n$-completeness
without specifying the reducibility $\leq_c$.  At certain times we will need to consider other reducibilities
on subsets of $\omega$, such as Turing reducibility $\leq_T$ or $1$-reducibility
$\leq_1$, and these reducibilities give rise to their own (analogous) notions
of completeness, but when referring to those notions we will always name them
specifically.

The standard notion of an $m$-reduction from one set to another is relevant here,
since computable reducibility on equivalence relations can be seen as one
of two natural ways to extend $m$-reducibility from sets to binary relations.
% The way familiar to most readers, e.g. from \cite{Soare:87}, is to treat the relation as a set of pairs. This gives the usual notion of $m$-reducibility:
%\begin{equation}\label{def:set-reducibility}
%A\le_m B\iff (\exists\text{~computable~} f)
%(\forall x,y) [(x,y)\in A\lra f(x,y)\in B].
%\end{equation}
%There is convenience in this approach as it allows any result about reducibility of sets to be readily extended to $n$-ary relations for any $n$.
However when comparing the complexity between equivalence relations, it is more natural to use computable reducibility. Following Definition \ref{defn:reduction}, we define:
\begin{equation}\label{def:rel-reducibility}
A\le_c B\iff \exists\text{~computable~} g \,
\forall x,y \, [ \la x,y \ra \in A\lra \la g(x),g(y)\ra\in B].
\end{equation}

The resulting reducibility has many interesting and, at times, surprising
properties. It can be readily seen that \eqref{def:rel-reducibility} is a stronger
property than $m$-reducibility: $A\le_c B$ via $g$ implies $A\le_m B$ via the function~$f$ given by $f(x,y) =  \la g(x),g(y) \ra$, so a relation
complete for a level of the arithmetic hierarchy in the sense of relations as in~\eqref{def:rel-reducibility} is also complete in the sense of~sets. The converse fails:  in Subsection~\ref{ssec:noPicomplete} we shall see that some levels do not have complete
equivalence relations at all.

It is   easy to see that  for every $n$, there   exists an
equivalence relation $E$ that is  $\Sigma^0_n$-complete under~$\leq_c$:  let $(V_e)\sN e$ be an effective list of all symmetric $\Sigma^0_n$ relations, let $(S_e)\sN e$ be the effective list of their transitive closure, and let $E = \bigoplus_e S_e $ be their effective sum. 
In Section~\ref{sec:Pi01complete} we demonstrate the existence of $\PI 1$-complete
equivalence relations, which is not nearly as obvious. We give a natural example from complexity, namely equality of quadratic time computable functions. In contrast,  we also show that for $n\geq
2$, no equivalence relation is $\Pi^0_n$-complete.

\subsection*{Preorders} Equivalence relations were the first context in which reductions arose,
but the notion of a reduction can be applied equally well to any pair
of binary relations:  The $E$ and $F$ in Definition \ref{defn:reduction}
need not be required to be equivalence relations.  In fact, a reduction
could be considered as a homomorphism of structures in the language
containing a single binary relation symbol, whether or not that symbol
defines an equivalence relation on either structure.  Preorders
(i.e.\ reflexive transitive binary relations $\preceq$, with $x\preceq y\preceq x$
allowed even when $x\neq y$) form a natural collection of binary relations
on which to study reductions, and in Section~\ref{sec:preorders} we will turn
to the question of computable reducibility on preorders,
with some interesting uses of effectively inseparable Boolean algebras.
In this way we   obtain some examples of natural
$\Sigma^0_n$-complete preorders  from
  computability theory.   For instance, almost inclusion $\sub ^*$ of r.e.\ sets, and weak truth table reducibility on r.e.\ sets, are $\SI 3$ complete preorders.  To obtain corresponding completeness  results for equivalence relations, we use the  easy Fact~\ref{fact:preorder} that completeness of a preorder $P$ at some level of the arithmetical hierarchy implies completeness of the corresponding symmetric fragment which is an equivalence relation.
Using this   we obtain, for instance,  the $\SI 3$ completeness  of several
of the equivalence relations considered in \cite{Coskey.Hamkins.Miller:12}.
In Section
\ref{sec:Sigmacomplete}  we show that Turing reducibilitiy on r.e.\ sets is a $\SI 4$-complete preorder, and hence Turing equivalence on r.e.\ sets is also $\SI 4 $-complete as an equivalence relation.  Using a direct construction, Fokina, Friedman and Nies \cite{Fokina.Friedman.ea:12} have proved that $1$-equivalence and $m$-equivalence on r.e.\ sets are $\SI 3$-complete.

Our results can be applied to obtain completeness results for preorders and equivalence relations from  effective algebra. For instance, Ianovski \cite{Ianovski:12} has proved the $\SI 3$-completeness of effective embeddability    for computable  subgroups of the additive group of rationals by reducing almost inclusion of r.e.\ sets. On the other hand, the
 completeness of  several natural equivalence relations, such as isomorphism of finitely presented groups, remains open. Such questions are discussed in the final section.

The articles already cited here are among  the principal references for  this subject.  For ordinary computability-theoretic questions,
we suggest \cite{Soare:87}, a second edition of which is to appear  soon.
 %The complexity of equivalence relations under this kind of reducibility has
%recently been studied by a number of authors from various disciplines. Within
%complexity theory it seems to have been first introduced in
%\cite{Fortnow.Grochow:11}, and later studied in \cite{Buss.Chen.Flum.Friedman.Muller:11}.

\section{Preliminaries} \label{s:prelim}  Let $X^{[2]}$ denote the unordered pairs  of elements of $X$. A set $R \sub X^{[2]}$ is \emph{transitive} if $\{u,v\}, \{v,w\} \in R$ and $u \neq w$ implies $\{u, w\} \in R$.  We view equivalence relations as transitive subsets of $\w^{[2]}$.   %We view the $p$-th c.e.\ set $W_p$ as a subset of $\w^{[2]}$.

 \begin{fact}\label{fact:preorder}
Let $P$ be a preorder, $E=\{ \{ x,y\} \colon \, \la x,y \ra \in P$ and $\la y,x\ra \in P\}$, and $n\in\omega$. If $P$ is complete among $\Sigma^0_n$ {\rm [$\PI n$]} preorders under $\leq_c$ then $E$ is complete among $\Sigma^0_n$ {\rm [$\PI n$]} equivalence relations under $\leq_c$.
\end{fact}
\begin{proof} We give the proof for $\SI n$. Note that $E$ is clearly $\Sigma^0_n$. Let $F$ be an arbitrary $\Sigma^0_n$ equivalence relation. In particular $F$ can be seen as  a preorder and so there is a computable function $f$ such that for every $x,y$, $ \{ x,y \} \in F$ iff $\la f(x),f(y)\ra \in P$. Then $f$ witnesses that $F\leq_c E$.
\end{proof}
In parts of the Subsections \ref{ss:Pi 1 complete} and \ref{ss:Preord} we will assume familiarity with the basic notions of computational complexity theory; see for instance \cite{Balcazar.Daz.ea:88}.
We consider languages $A \sub \Sigma^*$ for the  alphabet $\Sigma= \{0,1\}$. Recall that for a function $h \colon \omega \to \omega$, a language $A$ is \emph{computable in time $h$} if there is an $O(h)$ time-bounded multitape Turing machine (TM)  deciding membership in $A$. The class of such languages is denoted $\DTIME(h)$. Languages in $\DTIME(n)$ are also called \emph{linear time}.

	Recall that a function $h$ from numbers to strings with $h(n) \ge n$  is \emph{time constructible} if $h(n) $ can be computed in $O(h(n))$ steps on a multitape TM.
Suppose  $c \in \omega$ is a constant. If a function $h \colon \omega \to \omega $ is time constructible, then we can equip any given Turing machine with a ``clock'' on an extra tape. This is a   counter initialized at $c \cdot h(n)$ and decremented each step of the given machine. It  stops  the given machine on reaching $0$. Since $h$ is time constructible, the number of steps the clocked machine needs is still within $O(h)$.  Numbering  all Turing machines clocked in this way yields an effective listing  of the class  $\DTIME(h)$.

%%%%%%%%%%%%%%

%\input{Fragments/Andre_Pi01completeness}

\section{$\PI n$-completeness for  equivalence relations}\label{sec:Pi01complete}

In this section we show that there is a $\PI 1 $-complete equivalence relation, but no $\PI n$-complete equivalence relation  for $n \ge 2$.

\subsection{$\PI 1$ equivalence relations}  \label{ss:Pi 1 complete}

For each function $f \colon \, \omega \times \omega \to \omega$ and $x \in \omega$, let $f_x$ denote the function $n \mapsto f(x,n)$.
 If $f$ is computable,   $$E_f =\{ \la x,y \ra \colon \, f_x= f_y\}$$ appears to be  a   natural example of a $\PI 1 $ equivalence relation. We show that indeed  every $\PI 1 $ equivalence relation is obtained in this way. A slight extension of the argument yields a function $f$ such that  the corresponding equivalence relation $E_f$ is  $\PI 1 $-complete. We show that $f$ can in fact be chosen polynomial time computable.

For each $\PI 1$ equivalence relation $E$, there is a computable   sequence of cofinite relations  $(E_t)\sN t$  contained in $\omega^{[2]}$ with $E = \bigcap_ t E_t$   such that from $t$ we can compute a strong index for the complement of   $E_t$, we have $E_t \supseteq E_{t+1}$ and
\begin{equation} \label{eqn:local_transitive}  E_t \cap [0,t]^{[2]} \text{ is transitive.} \end{equation}

\begin{proposition}  \label{prop:repr_by_binary_function}  For each $\PI 1$ equivalence relation $E$, there is a computable binary  function $f$ such that $E=E_f$, namely, for each $x,y$, we have
\begin{equation} \label{eqn:eqrel_function}  xEy \lra f_x = f_y.\end{equation}  \end{proposition}

\begin{proof}  Let $(E_t)\sN t$ be as in (\ref{eqn:local_transitive}). We define  $f(x,n)$ by recursion on $x$. Let
$$f(x,n) =\displaystyle
         \begin{cases}
            f(r,n), & \text{if  $r< x$ for the  least $r$ such that $rE_{\max(x,n)}  x$,}\\
            x, & \text{otherwise.}
         \end{cases}
$$

Note that for each $z$,  for sufficiently large $k$, we have $f(z,k) = \min[z]_E$. Here $[z]_E$ is the equivalence class of $z$ with respect to $E$.
\n We verify that (\ref{eqn:eqrel_function}) is satisfied.

For the implication ``$ \to$'', we   show  the following by induction on~$x$.

\vsps

\n {\bf Claim.}  {\it For each $y<x$, for each $n$, if $ yE_{\max(x,n)} x$ then $f_y(n) = f_x(n)$. }

\vsps

\n To see this, let $r<x$ be as in the definition of $f(x,n)$. Then $r \le y<x$. By definition of $f$ we have  $f_x(n) = f_r(n)$. By transitivity in  (\ref{eqn:local_transitive}), we have $r E_{\max (x,n)} y$, and hence   $r E_{\max (y,n) } y$. Then by inductive hypothesis, $f_r(n) = f_y(n)$.

For the implication ``$ \leftarrow$'', suppose that  $y \lnot E x$.   Then $f_y(k) \neq f_x(k)$  for sufficiently large $k$ by the remark after the definition of~$f$ above.
\end{proof}

As a corollary we obtain a   presentation of $\PI 1$ equivalence relations as the uniform intersection of recursive ones.
\begin{cor} \label{cor:rep_Pi 1} For each $\PI 1$ equivalence relation $E$, there is a uniformly recursive sequence $(F_n) \sN n$ of equivalence relations such that $F_n \supseteq F_{n+1}$ and $E = \bigcap_n F_n$. \end{cor}
\begin{proof} Let $E= E_f$  for a computable function $f$. Let \bc $F_n = \{ \la x, y \ra \colon \, f_x \uhr n = f_y \uhr n\}$,  \ec
where $g\uhr n$ denotes the tuple $\la g(0), \ldots, g(n-1)\ra$. \end{proof}
We obtain a $\PI 1$-complete equivalence relation by exposing a uniformity in the proof of Proposition~\ref{prop:repr_by_binary_function}.
\begin{theorem} \label{thm:Pi1_complete} There is a computable binary function $g$ such that the equivalence relation
  $E_g =\{ \la x,y \ra \colon \, g_x= g_y\}$ is $\PI 1 $-complete. \end{theorem}
\begin{proof} Let $E^i = \omega^{[2]} - W_i$ where the $i$-th r.e.\ set $W_i$ is viewed as a subset of $\omega^{[2]}$. Uniformly in $i$ we may obtain a computable  sequence of strong indices $(E^i_t)\sN t$ as the one used in the proof of Proposition~\ref{prop:repr_by_binary_function}  such that $E^i = \bigcap_t E^i_t$ whenever  $E^i$ is transitive. Now define $g(\la i,x\ra,n)$ to have the value $f(x,n)$ as above where $E=E^i$. By the argument in that proof, if $E^i$ is transitive then $E^i \le_c E_g$ via the function $x \mapsto \la i,x\ra$.
\end{proof}

We can in fact obtain   in Theorem~\ref{thm:Pi1_complete}  a function on binary strings that is   quadratic time computable (see Section~\ref{s:prelim}).
We thank Moritz M\"uller at the  Kurt G\"odel Research Institute in  Vienna  for suggesting a simplified proof which we provide at the end of this subsection.

\begin{lemma}\label{lem:Mueller}
	For each computable  binary function $g$, there is a quadratic time computable binary function $G$ defined on strings over $\{0,1\}$,  and a computable unary function $p$ mapping numbers to strings,  such that $g(x,n) = G(p(x), p(n))$, and $G(w,v)=0$ for any string $v$ not of the form $p(n)$.
\end{lemma}
Unary quadratic time functions over the alphabet $\{0,1\}$ are given by indices for multitape  Turing machines (TM)  that are equipped with a counter forcing them to stop  in time $O(n^2)$, where $n$ is the input length.
%; that is, multitape TM equipped with a counter on a specified tape so that the number of steps
\begin{theorem}\label{thm: quadratic time} (i) There is a quadratic time computable binary function $G$ such that $E_G$ is $\PI 1 $-complete. (ii)~Equality of unary quadratic time computable  functions is a $\PI 1$-complete equivalence relation.
\end{theorem}

\begin{proof}[Proof of Theorem] 	\emph{(i)} Let  $G$ be the   function obtained from the function  $g$ of Theorem~\ref{thm:Pi1_complete} according to the lemma.  Let $p$ be as in the lemma. Then $x E_g y \lra G_{p(x)} = G_{p(y)}$. Since $E_g$ is complete, so is $E_G$.
	
\n	\emph{(ii)} Since all the Turing machines considered stop on every   input in a quadratic time bound, the equivalence relations is $\PI 1$. The  completeness follows from (i)  since from a string $w$ we can compute an index for a   quadratic time TM computing $G_w$.
\end{proof}

\begin{proof}[Proof of Lemma~\ref{lem:Mueller}]
	Let $g(x,y)$  (with $x,y$ represented in unary) be computable in time $t(x,y)$ where  $t$ is a computable function that is  increasing   in both  $x$ and  $y$.
We want a time constructible $h$ satisfying $t(n,k)\le h(n)+h(k)$. This can be achieved by having $h(n)$ run $g(n,n)$ and counting the number of steps.
	This satisfies $t(n,k)\le h(n)+h(k)$ as for $n\le k$, $t(n,k)\le t(k,k)=h(k)$, and $h(n)=m$ can be calculated in time quadratic in $m$: Place a binary counter
	at the start of the tape and simulate $g$ to the right of it.
	In the worst case scenario for every step of $f$ the machine would iterate over
	$m+\log m$ cells, and whenever the counter grows in size the entire tape
	would need to be shifted right which in total would take $m\log m$,
	coming to $m\log m+m(m+\log m)$ or $O(m^2)$.

	To compute $G(a,b)$, verify whether $a=1^x01^{h(x)}$ and $b=1^n01^{h(n)}$. If so
	output $g(x,n)$, else output 0.
	Observe that $G$ is quadratic time: We can verify whether $a=1^x01^z$ is of the form
	$1^x01^{h(x)}$ by beginning to compute $g(x,x)$, stopping whenever the
	number of steps exceeds $z$.
	If the input is of the right form we can compute $g(x,n)$ in time $t(x,n)\le
	h(x)+h(n)= O(|a|+|b|)$.

	Note that  for every $x,n$,
	\bc  $G(1^x01^{h(x)},1^n01^{h(n)})=g(x,n)$, \ec
	and for
	all other values the function is 0. Thus the function $p(y) =1^y01^{h(y)} $ is as required.
	\end{proof}

	In work independent of ours \cite[Proposition 3.1]{Cholak.etal:2011} the authors
	offer a characterisation of $\PI 1$ preorders:
	Every such preorder is computably isomorphic to the
	inclusion relation on a uniformly computable family of sets. As an equivalence relation is
	merely a symmetric preorder, a $\PI 1$ equivalence relation is then isomorphic
	to equality on such a family. Since sets can be viewed as 0/1 valued functions, this yields another proof of Proposition~\ref{prop:repr_by_binary_function}. Methods similar to the ones employed in the proof of Theorem~\ref{thm: quadratic time} above now  yield a version of that Theorem for the  inclusion of  quadratic time computable languages. 	
	\begin{thm}[\cite{Ianovski:12} based on \cite{Cholak.etal:2011}] \label{thm:Quad_inclusion} Inclusion of  quadratic time computable languages is a complete $\PI 1$ preorder. \end{thm}
	For a proof  see \cite[Thm 5.5]{Ianovski:12}.
	Using that  result one can also obtain a stronger form of Cor.\ \ref{cor:rep_Pi 1} where the $F_n$ have at most $2^n$ equivalence classes.

%\input{Fragments/Egor_Pi01trees}

%%%%%%%%%%%%%%
\subsection{For $n\geq 2$ there is no  $\PI n$-complete equivalence relation }
\label{ssec:noPicomplete}

$\PI n$~equivalence relations often  occur naturally. For instance, equality of r.e.\ sets as a relation on  indices is $\PI 2$. See \cite{Coskey.Hamkins.Miller:12} for further examples along these lines.
Computable model theory is a further source of interesting examples.  Consider the class $\+ P$ of computable permutations on $\omega$ with all cycles finite. Isomorphism as a relation on computable indices for structures in this class is $\PI 2$. This equivalence relation is easily seen to be computably equivalent to equality on r.e.\ sets.  Melnikov and Nies \cite{Melnikov.Nies:13}  have studied isometry of compact computable metric spaces, They showed  that it is properly $\PI 3$ in general, and properly $\PI 2$  for effectively compact metric  spaces.

Here we show that there is no $\PI n$-complete equivalence relation  for $n \ge 2$. Intuitively speaking, one would then expect that  the computable reducibility degrees of    $\PI n$ equivalence relations are   complicated.

\begin{thm}
\label{thm:noPi2complete}
For each $\Pi^{0}_2$ equivalence relation $E$ there is a $\Delta^{0}_2$ equivalence relation $L$ such that $L\not\leq_{c}E$.
\end{thm}
Let $\aaa$ be a computable ordinal. Relativizing to  $\emptyset^{(\aaa)}$ yields an immediate corollary. We say that $E \leq_{\boldsymbol{0}^{(\aaa)}} F$ if Definition \ref{defn:reduction} holds for $E$ and $F$ for some $f\leq_T \emptyset^{(\aaa)}$.
\begin{cor}
\label{cor:noPincomplete}
Let $\aaa$ be a computable ordinal. No $\Pi^0_{\aaa+2}$ equivalence relation
can be hard for all the $\Delta^0_{\aaa+2}$ equivalence relations
under the  reducibility $\leq_{\boldsymbol{0}^{(\aaa)}}$,
let alone under computable reducibility $\leq_c$.
\qed\end{cor}
  In contrast, there is a $\Pi^1_1$-complete equivalence relation $E$. Similar to the case of $\SI n$   relations,   the class of symmetric $\Pi^1_1$ relations is  uniformly  closed under taking the transitive closure.  Let $(V_e)\sN e$ be an effective list of all symmetric $\Pi^1_1$ relations, let $(S_e)\sN e$ be the effective list of their transitive closure, and let $E = \bigoplus_e S_e$.

\begin{proof}[Proof of Theorem \ref{thm:noPi2complete}]

$E$ be any
$\Pi^{0}_2$ equivalence relation on $\omega$. Since $E$ is $\Pi^{0,\ES'}_1$
we may fix a $\ES'$-computable function $f$ such that $\la x,y\ra\in E$ iff $f(x,y,t)=1$ for every $t$. Now define $L$ as follows.

For $e_0\neq e_1$, $\{ \langle e_0,s_0\rangle,\langle e_1,s_1\rangle\} \not\in L$ for any $s_0,s_1$. For each $e$, we distinguish two cases.

\n {\it Case 1.}  There is a number $t$, chosen least,   such that \bc  $\varphi_e(\langle e,0\rangle)\downarrow=x$, $\varphi_e(\langle e,1\rangle)\downarrow=y$, $\varphi_e(\langle e,t+2\rangle)\downarrow=z$, and $f(x,y,t)=0$. \ec Let  $v$ be  least such that $f(x,z,v)=0$ or $f(y,z,v)=0$.

\vsps

\bi \item If $f(x,z,v)=0$ we  set $\{\langle e,0\rangle,\langle e,t+2\rangle \}\in L$.
\item
If
$f(y,z,v)=0$
we  set $\{\langle e,1\rangle,\langle e,t+2\rangle \}\in L$. \ei
 Declare $\{\langle e,s_0\rangle,\langle e,s_1\rangle\}\not\in L$ for every other pair $s_0\neq s_1$.

\n {\it Case 2.} There is   no such  number $t$. Then we set $\{\langle  e,s_0\rangle,\langle e,s_1\rangle \} \not\in L$ for every $s_0\neq s_1$.

To see  that $L$ is $\DII$, it suffices to note that
for each $e$, if Case 1 applies  then by the transitivity of $E$ the number $v$ exists. Transitivity of $L$  follows from the fact that we relate each element to at most one other element.

 Now suppose that $L\leq_c E$ via the total function $\varphi_e$. Since $\{\langle e,0\rangle,\langle e,1\rangle\}\not\in L$ it follows that the least number $t$ described above exists. By construction we ensure that either $\{\langle e,0\rangle,\langle e,t+2\rangle\}\in L$ and $\{x,z\}\not\in E$, or we have $\{\langle e,1\rangle,\langle e,t+2\rangle\}\in L$ and $\{y,z\}\not\in E$, a contradiction.
\end{proof}

\subsection*{Remark}
Observe that given any two $\Pi^0_{n+2}$ equivalence relations $R,S$, the disjoint union $R\sqcup S=\{\{2n,2m\}\mid n R m\}\cup \{\{ 2n+1,2m+1\} \mid n S m\}$ is also a $\Pi^0_{n+2}$ equivalence relation and we have $R,S\leq_c R\sqcup S$. Hence Theorem \ref{thm:noPi2complete} shows that amongst $\Pi^0_{n+2}$ equivalence relations there can be no maximal element under $\leq_c$.

The following is an immediate consequence of Theorem \ref{thm:noPi2complete}; for preorders we also apply  Fact~\ref{fact:preorder}.
\begin{cor} For each $n \ge 2$, there is no $\PI n$-complete preorder, and no $\Delta^0_n$-complete equivalence relation, or preorder. \end{cor}

There also is no $\PI 2$  analogue of Cor.\ \ref{cor:rep_Pi 1}. Let us say a $\PI 2$  equivalence relation $E$  is \emph{effectively $\PI 2$} if there is a uniformly r.e.\  sequence $(F_n) \sN n$ of equivalence relations such that  $E = \bigcap_n F_n$. (We may also require $F_n \supseteq F_{n+1}$ without loss of generality.)   If $S$ is an r.e.\ nonrecursive set then the $\DII$ equivalence relation with the two classes $S$ and $\omega-S$ is not effectively $\PI 2$.  Note that we can effectively list the effectively $\PI 2$ equivalence relations. By putting together all the effectively $\PI 2$ equivalence relations in the obvious way we see that there is a complete one under computable reducibility.

%\input{Fragments/Russell_NoPi2complete}

%%%%%%%%%%%%%%

%\input{Fragments/Russell_Sigma}

%\input{Fragments/Andre_Boolean_algebras}

%\input{Fragments/Andre_preorders_via_eff_insep}

\section{$\SI k$-complete preorders}
\label{sec:preorders}

The concept of an effectively inseparable   r.e.\  Boolean algebra was  introduced by  Pour-El  and  Kripke \cite[Lemmas 1,2]{Kripke.PourEl:67} when they studied the complexity of logical derivability for recursively axiomatizable theories.  The method was developed in more generality by  Montagna and Sorbi
\cite{Montagna.Sorbi:85}. For instance, in their Theorem 3.1 they showed that the preorder associated with  any effectively inseparable    Boolean algebra is $\SI 1$-complete.   They developed their method further  in order to show that for  a  sufficiently strong arithmetical theory $T$  (such as  a consistent  axiomatizable extensions of PA), for any r.e.\  preorder $\preceq$ there is a $\Sigma_1 $ formula $G$ with $i \preceq j $ iff $T \vdash G(\underline i) \to G(\underline j)$.

We are interested mostly  in semantic preorders  such as reducibilities, which are usually  at higher levels of the arithmetical hierarchy. So  we extend the methods discussed above  in order to  show the   completeness of a  $\SI k$ preorder, using  effectively inseparable  $\SI k$ Boolean algebras. Thereafter we apply it to preorders in the subrecursive setting, which are $\SI 2$,  and to  preorders on the r.e.\ sets arising from computability theory, which are $\SI 3$. Each time,  we naturally embed  an effectively inseparable  $\SI k$ Boolean algebra for the appropriate $k$  into the preorder.

For sets  $X, Y$, we write $X \sub^*Y$ ($X$ is  almost contained in  $Y$) if $X \setminus Y$ is finite. We write $X = ^* Y $ if $X \sub^* Y \sub^* X$. The preorders we will  consider in  this section for the cases $k=2, 3$ are  given  on certain natural classes of sets,  either by almost containment, or by reducibilities,.

\subsection{The method of effectively inseparable $\SI k$ Boolean algebras}

Our proofs rest on
the  following notion, which  is an analogue of creativity  for disjoint pairs of sets. For detail see \cite[II.4.13]{Soare:87}. We     relativize r.e.\ sets to $\ES^{(k-1)}$, thereby obtaining $\SI k$ sets.  Note, however, that the analogue~$g$ of the  productive function remains computable.
\begin{definition}  (see \cite[II.4.13]{Soare:87})  \label{def:Insep_Sk} We say that disjoint $\SI k$  sets $A, B \sub \omega$ are \emph{effectively inseparable} (e.i.) if there is a computable binary function  $g$ (called a productive function) such that for each $p,q \in \omega$, \bc if   $A\sub W_p^{\ES^{(k-1)}}$,  $B\sub W_q^{\ES^{(k-1)}}$ and $W_p^{\ES^{(k-1)}} \cap W_q^{\ES^{(k-1)}} = \ES$,

then $g(p,q) \not \in W_p^{\ES^{(k-1)}} \cup W_q^{\ES^{(k-1)}}$. \ec
\end{definition}
\begin{remark}	 \label{rem:expl_Insep_Sk}  {\rm
(a) If we enlarge both components of the disjoint pair to disjoint~$\SI k$  sets, the function $g$ is still a productive function.

\n (b) We say that  a disjoint pair   $A,B$ of $\SI k$ sets       is     \emph{$m$-complete for disjoint pairs}    if  for each disjoint pair  $U,V$   of $\SI k$ sets, there is a computable  map $\aaa$ such that  \bc $x \in U \lra \aaa(x) \in A$, and $x \in V \lra \aaa(x)  \in B$.  \ec  It is a well-known result going back to Smullyan~\cite{Smullyan:61}   that each e.i.\ pair of $\SI k$ sets is $m$-complete in this sense.  See \cite[II.4.15]{Soare:87} for a more recent reference. Note that the $m$-reduction  in \cite[II.4.15]{Soare:87} can be determined from   indices of the    given  $\SI k$ sets $U,V,A,B$, and   the productive function $g$ for $A,B$. Even If $U \cap V \neq \ES$ or $A \cap B \neq \ES$, as long as $g$ is total the construction still yields  a  (total) computable function.}
\end{remark}

 % Clearly such  pair  $A,B$ is e.i.
	
%	 By (2), mere forward implications $\to$ are already sufficient: let $U,V$ be an e.i. pair, and enlarge  it to  $f^{-1}(A)$, $f^{-1}(B)$ if necessary. }

We consider Boolean algebras in the language with partial order $\le$, meet $\wedge$, join $\vee$, and complementation $\, '$.
A $\SI k$ Boolean algebra  $\+ B$  is represented by a model $(\omega, \preceq, \vee, \wedge, ' )$ such that $ \preceq $ is a $\SI k$ preorder, $ \vee, \wedge$ are   computable binary functions, $\, '$ is a computable unary function, and the quotient structure $ (\omega, \preceq, \vee, \wedge, ' )/_{\approx}$ is isomorphic to $\+ B$. Here $\approx$ is the equivalence relation corresponding to $\preceq$, and we assume the functions are compatible with $\approx$.  We may also  assume that $0 \in \omega$ denotes the least element of the Boolean algebra, and $1 \in \omega$ denotes the greatest element. If necessary, we write $\preceq_\+ B$ etc.\ to indicate the Boolean algebra a relation on $\omega$ belongs to.

 The  usual effectiveness notions defined for sets or functions on the natural numbers can be transferred  to  $\+ B$. For instance, we say that an ideal $I$ of $\+ B$ is $\SI k$ if $I$ is $\SI k$  when viewed as a subset of $\omega$.

Let $\+ F$ be a computable   Boolean algebra  that is   freely generated by a computable sequence  $(p_n)\sN n$. For instance, we can take as $\+ F$ the finite unions of intervals $[x,y)$ in $[0,1)_\QQ$. We fix  an  effective encoding of $\+ F$ by natural numbers.  Then, equivalent to the definition above,  a  $\SI k$ Boolean algebra is given in the form $\+ F / \+ I$,  where $\+ I$ is a $\SI k$ ideal of $\+ F$.
We rely on this view  for  coding a $\SI k$ preorder into a $\SI k$ Boolean algebra.  We  slightly  extend concepts and results    of~\cite{Montagna.Sorbi:85} where $k=1$.

\begin{lemma} \label{lem:Represent_Sigma_k}    For any $\SI k$ preorder $\preceq$,  there is a $\SI k$ ideal $\+ I$ of~$\+ F$ such that
	 $ n \preceq k \leftrightarrow  p_n - p_k  \in \+ I $.
\end{lemma}

\begin{proof} Let $\+ I$ be the  ideal  of $\+ F$ generated by $\{ p_n - p_k \colon n \preceq k \}$.
	The implication ``$\rightarrow$''   follows from the  definition. For the implication ``$\leftarrow$'', let $\+ B_\preceq$ be the Boolean algebra generated by the subsets of $\omega$  of the form  $\hat i = \{ r \colon \, r \preceq i\}$. The map $p_i \mapsto \hat i$ extends to a Boolean algebra homomorphism $g \colon \, \+ F \to \+ B_\preceq$ that sends $\+ I $ to $0$. If $n \not \preceq k$ then $\hat n \not \subseteq \hat k$, and hence $p_n - p_k \not \in \+ I$.
\end{proof}

%Similarly, a sequence $(x_i)\sN i$ of elements of $\+ B$ is computable  if it is computable when viewed as a sequence of natural numbers; more precisely, there is a computable sequence $(n_i)\sN i$ such that $x_i= n_i/_\approx$.

\begin{definition}\label{def:Insep_Sk_BA}
	We say that a $\SI k$ Boolean algebra $\+ B$ is  \emph{$\SI  k$ effectively inseparable} (e.i.)  if the sets $[0]_\approx$ and $[1]_\approx$ (that is, the names for $0 \in \+ B$ and for $1\in \+ B$, respectively) are effectively inseparable.
\end{definition}

We will frequently use the following criterion to show that a $\SI k$ Boolean algebra $\+ B$ is effectively inseparable.

\begin{fact} \label{fa:easy e.i.}   Suppose that  $U,V$ is a pair of  e.i.\ $\SI k$ sets. Let    $\+ B$  be a $\SI k$ Boolean algebra.   Suppose there is a computable function $g$ such that \bc $x \in U \to g(x) \approx 0$ and $x \in V \to g(x) \approx 1$. \ec
	Then $\+ B$ is effectively inseparable. \end{fact}
\begin{proof} Let $A = \{x \colon \,  g(x) \approx 0\} $ and $B = \{x \colon \,  g(x) \approx 1 \}$. Then $A,B$ are $\SI k$ sets,  $A \cap B = \ES$, $U \sub A$, and $V \sub B$. By Remark~\ref{rem:expl_Insep_Sk} this implies that $A,B$ are effectively inseparable as $\SI k$ sets, whence $\+ B$ is  effectively inseparable as a $\SI k$ Boolean algebra.	
\end{proof}
We slightly extend a result   \cite[Prop.\ 3.1]{Montagna.Sorbi:85}. As mentioned already, that work  goes back to results  of  Pour-El  and Kripke \cite[Lemmas 1,2]{Kripke.PourEl:67} who only worked in  the setting of  axiomatizable theories. It is worth   including   a proof of the extension  because we only partially relativize the  setting of \cite{Montagna.Sorbi:85}: The preorder of the Boolean algebra is $\SI k$, while  all of  its algebraic operations, as well as the reduction function in the definition of completeness,  are still computable.

\begin{thm}\label{pro:Insep_are_complete} Suppose a Boolean algebra $\+ B$ is  $\SI k$-effectively inseparable.   Then $\preceq_\+ B$ is a $\SI k$-complete  preorder.
\end{thm}

\begin{proof}   We say that a  $\SI k$ Boolean algebra  is    \emph{$\SI k$-complete} if each $\SI k$ Boolean algebra $\+ C$  is computably embedded into it. We first show that   the given Boolean algebra $\+ B$ is   $\SI k$-complete.    For an element $z$ of a Boolean algebra define $z^{(0)}= z$ and $z^{(1)} = z'$.

   Given a  $\SI k$ Boolean algebra $\+ C$, we  construct the desired computable embedding $h$ of $\+ C$  into $\+ B$  recursively.
Recall that the domain of a presentation of $\+ C$  is $\omega$ by our definition. Suppose $y_i= h (i)$ has been defined for $i< n$.  For each $n$-bit string $\tau$ let
\bc $p_\tau = \bigwedge_{i< n}  i^{(\tau_i)} $  and $y_\tau =  \bigwedge_{i< n} y_i^{(\tau_i)}$.  \ec
In particular, if $n=0$, we let $p_\estring = 1_\+ C$ and $y_\estring = 1_\+ B$.

We define $h$ in such a way that   for each $n$ and each  string $\tau$ of length $n$, we have
\begin{equation} p_\tau \approx_\+ C 0 \lra y_\tau \approx_\+ B 0. \tag{$*$} \end{equation}
This is clearly the case for $n= 0$.  Inductively assume ($*$) holds for $n$.
To define $h(n)$ we use the following.

\n {\bf Claim.}  {\it From  a string $\tau$ of length $n$ we can effectively determine an element $z_\tau \preceq_\+ B y_\tau$  such that }
\bc $n \wedge p_\tau \approx_\+ C 0 \lra z_\tau \approx_\+ B 0 $,  and $n \succeq_\+ C p_\tau \lra z_\tau \approx_\+ B y_\tau $.  \ec
\n To see this, first suppose $y_\tau \not \approx_\+ B  0$.
Then the  pair of $\SI k$ sets
\bc $A  = \{r \colon \, r \wedge y_\tau \approx_\+ B 0\}$, $B = \{ r \colon \, r \succeq_\+ B y_\tau \}$  \ec
is    disjoint. By  Remark~\ref{rem:expl_Insep_Sk}(a)  this  pair  is   e.i.\ with the same productive function~$g $ as the one for the    e.i.\ pair  of sets $\{r \colon \, r   \approx_\+ B 0\}$, $ \{ r \colon \, r \approx_\+ B 1_\+ B \}$.

Since $y_\tau \not \approx_\+ B  0$, by ($*$) for $n$  we have $p_\tau  \not \approx_\+ C  0$. Hence the  pair of  $\SI k$ sets
\bc $U = \{k \colon \, k \wedge p_\tau \approx_\+ C 0\}$, $V = \{ k \colon \, k \succeq_\+ C p_\tau \}$  \ec is also disjoint. By  Remark~\ref{rem:expl_Insep_Sk}(b), we are uniformly given an $m$-reduction $\alpha$ from $U,V$ to $A,B$.   We let $z_\tau = y_\tau \wedge \aaa(n )$. Then the claim is satisfied in case $y_\tau \not \approx_\+ B  0$.

 If $y_\tau   \approx_\+ B  0$ then, by ($*$) for $n$,  we have   $p_\tau    \approx_\+ C  0$. None of the pairs of sets is disjoint now, but (based on the productive function $g$) we still have a computable index for a   function $\alpha$,  and hence  a definition of    $z_\tau$ such that $z_\tau  \approx_\+ B 0$, which vacuously satisfies the claim.

Now   let  $h(n) = \bigvee_{|\tau|=n} z_\tau$. It is clear that ($*$) holds for $n+1$. Then by induction,  the map $h$   is a computable embedding of $\+ C$ into $\+ B$, as required.

To conclude the  proof of the theorem, let $\preceq_P$ be  any $\SI k$ preorder.  	By Lemma~\ref{lem:Represent_Sigma_k}   there is a $\SI k$ ideal $\+ I$ of the free Boolean algebra $\+ F$ such that
		 $ n \preceq k \leftrightarrow  p_n - p_k  \in \+ I $. Since $\+ B$ is $\SI k$-complete, there is a computable embedding $g$ of $\+ F/ \+ I$ into $\+ B$. Thus, $ n \preceq_P k \leftrightarrow g(p_n) \preceq_\+ B g(p_k)$.
\end{proof}

 A forth-and-back version of the argument above shows that any two e.i.\ $\SI k$ Boolean algebras are effectively isomorphic. This fact for $k=1$ was  already noted in~\cite{Kripke.PourEl:67,Montagna.Sorbi:85}.  It  is interesting in view of the applications below: for a fixed~$k$, all the preorders considered are complete ``for the same reason".

%%%%%%%%%%%%%%{Applications}

\subsection{Derivability in first-order logic}

Let $T$ be a recursively axiomatized sufficiently strong theory in the language of arithmetic, such as Robinson arithmetic $Q$. Note  that such a theory  $T$ can be finitely axiomatizable. The following  was first observed in  \cite[Section 4]{Montagna.Sorbi:85}, building on    Pour-El  and Kripke \cite{Kripke.PourEl:67}.

\begin{theorem}\label{thm:PA}  For sentences $\phi, \psi$ in the language of $T$, let $\phi \preceq \psi$ if  $\vdash_T \phi \to \psi$. Then $\preceq $  is a $\SI 1 $-complete preorder.
\end{theorem}

\begin{proof}
	Clearly $\preceq$ is $\SI 1$. It was first observed by Smullyan~\cite{Smullyan:61} that  the Lindenbaum algebra of $T$ is $\SI 1$-effectively inseparable (or see \cite[Section 4]{Montagna.Sorbi:85}). Now Theorem~\ref{pro:Insep_are_complete} implies that $\preceq$ is complete.
\end{proof}

 In \cite{Ianovski:12} the  first author has given a direct  proof that logical implication for the first order language in  a full signature is $\SI 1 $-complete. His proof   relies on a   coding of Turing  machine computations.

%%%%%%%%%%%%%%%%
\subsection{$\SI 2$   preorders in the resource-bounded setting} \label{ss:Preord} For relevant notation from  complexity theory  see Section~\ref{s:prelim}. In particular, there is an effective listing  of the class  $\DTIME(h)$, which we fix  below without further mention.
We use a technical tool:
\begin{lemma}\label{lem:pair_Sigma_2} Given a disjoint pair of $\SI 2$ sets $U, V$, we can effectively in $x$  determine a linear time computable language $L_x$ such that \bc $x\in U \to L_x \text{ is finite}$, and $x\in V \to L_x \text{ is cofinite}$. \ec
\end{lemma}
%[Proof of lemma]
\begin{proof} It follows from \cite[pg.\ 66]{Soare:87} that there is   a uniformly r.e.\  pair of  sequences  $  (S_x)\sN x, (T_x)\sN x$ such that   \bc $x\in U \to S_x \text{ is finite}$, and $x\in V \to T_x \text{ is finite}$. \ec To define $L_x$,  at stage $n$,  in linear time we determine whether $w \in L_x$ for each string $w$ of length $n$. Let $t\le n$ be largest  such that $t>0$ and     in $n$ steps one  can verify that (a) $S_{x,t} \neq S_{x,t-1}$ or  (b) $T_{x,t} \neq T_{x,t-1}$; if there is no such $t$, let $t=0$. If (a)    applies or $t=0$, declare that  $w$ is in  $L_x$. Otherwise, declare that  $w$ is not in $L_x$.
	
	Clearly $L_x$ is computable in linear time, uniformly in $x$. If $x \in U$, then $S_x$ is finite, so $T_x$ is infinite. Then   for almost all $n$  we are in case (b), so $L_x$ is finite.   If $x\in V$ then $T_x$ is finite, so $S_x$ is infinite, so for almost all $n$  we are in case (a) and $L_x$ is cofinite.
\end{proof}

We obtain a variation on  Theorem~\ref{thm:Quad_inclusion}.

\begin{theorem}\label{thm:polytime_inclusion}   The preorder of almost inclusion $\sub ^*$  among  quadratic time computable languages is $\SI 2$-complete.
\end{theorem}

\begin{proof} Let $(G_i)\sN i$ be an effective listing of all the linear time computable languages. Fix  a set $A$ computable in quadratic time but not in linear time. We claim that  the $\SI 2$ Boolean algebra  \bc  $\+ B =  ( \{A\cap L \colon \,  L \text{ is   linear time}\}, \sub ^*)/=^*$ \ec  with the canonical representation given by $(G_i)\sN i $, is $\SI 2$ effectively inseparable.
	To see this, in Lemma \ref{lem:pair_Sigma_2} let $U,V$ be any pair of $\SI 2$ effectively inseparable sets. If $x \in U$ then $A \cap L_x =^* \ES$. If   $x \in V$ then $A \cap L_x =^* A$. Now we apply
Fact~\ref{fa:easy e.i.}.

The preorder of  $\+ B$ on indices for linear time computable languages is $\SI 2 $-complete by Theorem~\ref{pro:Insep_are_complete}. From $i$ we can compute an index for $A \cap G_i$ within our effective listing of the quadratic time sets. Hence we can effectively
reduce the  preorder on $\+ B$ to  almost inclusion $\sub ^*$  among  quadratic time computable languages.
\end{proof}

In the following, $\le^p_r$ will  denote a polynomial time reducibility in between polynomial time  many-one (m) and Turing (T) reducibility. Clearly this reducibility is $\SI 2$ on  $\DTIME(h)$.
\begin{thm} \label{thm:polytime_reducibility}
Suppose the function $h$ is  time constructible and dominates all polynomials. Then   $\le^p_r$ on sets in $\DTIME(h)$ is    $\SI 2$-complete  preorder.
\end{thm}
\n For instance, we can let $h(n)=2^n$, or $h(n)= n^{\log\log n}$.
The proof relies on the following notion of  Ambos-Spies \cite{Ambos:86}.

\begin{definition}
Let 	 $f \colon \, \omega \to \omega$  be  a  strictly increasing,
	time constructible
	function. We say that a language
$A \sub \{0\}^*$ is \emph{super sparse}  via $f$ if
 $A \sub \{ 0^{f(k)}: k \in \omega \}$
and    ``$ 0^{f(k)} \in A$ ?''
can be determined in time
$O(f(k+1))$.
\end{definition}
Supersparse sets exist in the time classes we are interested in.

\begin{lemma}[\cite{Ambos:86}] \label{ssplemma} Suppose that $h: \omega
\to \omega$ is
an increasing time constructible function
with $\PP \subset \DTIME(h)$, so that $h(n) \ge n+1$ and $h$
eventually dominates all polynomials. Then there is a  {super sparse}
language  $A\in \DTIME(h)- \PP$.
\end{lemma}

\begin{proof}[Sketch of Proof]  Let $f(n) = h^{(n)}(0)$. Since $h$ eventually
dominates all polynomials, we can construct $A \sub \{0^{f(k)}: k \in
\omega\}$ such that $A \in \DTIME(h)$, but still diagonalize against all
polynomial time machines.
\end{proof}

We will keep this set $A$ fixed  in what follows. Given a reducibility $\le^p_r$, we    write $\mathbf a = \deg^p_r(A)$ for  the degree of $A$. We write $[\mathbf 0, \mathbf a]$ for the initial segment of the degrees below  $\mathbf a$, viewed as a partial order.

\begin{proof}[Proof of Theorem~\ref{thm:polytime_reducibility}]

Recall the e.i. $\SI 2$ Boolean algebra $\+ B$ from the proof of Theorem~\ref{thm:polytime_inclusion}. It is represented via an effective listing $(G_i)$ of linear time computable languages. We will embed a $\SI 2$ quotient of $\+ B$ into   $[\mathbf 0, \mathbf a]$.

Since $A$ is super sparse,  by  \cite[Lemma 6.2.9]{Nies:habil} for polynomial time sets $U, V$ we have \bc  $ A \cap U \le^p_r A \cap V  \leftrightarrow A \cap (U- V) \in \PP$. \ec
So we may well-define a map $\Phi \colon \+ B \to [\mathbf 0, \mathbf a]$ as follows. For a linear time set $L= G_i$, let \bc $\Phi(A \cap L/=^* ) = \deg^p_r( A \cap L)$.   \ec

Let $\+ I$ be the $\SI 2$ ideal of $\+ B$ consisting of all the equivalence classes   $A \cap L/=^* $ such that $ A \cap L \in  \PP$. Then $\+ C= \+ B / \+ I $ is in a natural way a $\SI 2$ Boolean algebra, and $\+ C$ is  effectively embedded into $[\mathbf 0, \mathbf a]$. Since $\+ B$ is e.i.\  and the map $\Phi$ is effective, it follows that  $\+ C$ is e.i. Hence by Theorem~\ref{pro:Insep_are_complete}, the $\SI 2$ preorder $\{\la e, i \ra \colon A \cap G_e \le^p_r A \cap G_i\} $ is complete. Via a computable    transformation  of indices,  this preorder is effectively reducible to $\le^p_r$ on $\DTIME(h)$.
	\end{proof}

%%%%%%%%%%%%OUT

\iffalse
A polynomial time 1-$tt$ reduction of $X$ to $Y$ is a polynomial  time Turing reduction
where in each  computation at most one oracle question is asked.

\begin{definition} \label{1tt} $X \le^1_{tt}Y$ if there are polynomial time
  computable functions $g: \Sigma^{<\omega} \times \{0,1\} \mapsto
  \{0,1\}$ and $h:\Sigma^{<\omega} \mapsto \Sigma^{<\omega}$ such that

$$ \fa w \in \Sigma^{<\omega} \ [X(w) = g(w, Y(h(x)))].$$
\end{definition}

\begin{lemma}[\cite{Ambos:86}] \label{Ambos} Suppose the language  $A$ is
  super sparse. Then the polynomial time T-degree
of any set $B \pT A$ consists of  a single 1-$tt$--degree.
 \end{lemma}

Following \cite[Section 6.2]{Nies:habil}, let $\+ B(\mathbf a)$ denote the complemented elements in $[\mathbf 0, \mathbf a]$. Let $(P_e)\sN e$ be an effective listing of the polynomial time computable languages.
\fi

%\input{Fragments/Andre_Sigma02Preorders}
%\input{Fragments/Andre_star_equivalence}

\subsection{Almost inclusion of r.e.\ sets}

\begin{theorem}
	\label{thm:starequal}  The preordering  $\{\la e, i \ra \colon W_e \sub^* W_i\}$  of almost inclusion  among  r.e.\   sets is $\SI 3$-complete.
\end{theorem}
\begin{proof}
	All sets in this proof will be r.e. Fix a non-recursive r.e.\ set $A$.  By $  X \sqcup Y =A$ we denote that r.e. sets $X,Y$ are a splitting of $A$. That is, $X \cap Y = \ES$ and $X \cup Y = A$.
	We also write $X \sqsubseteq A$ to  denote that $X$ is part of a splitting of $A$.
	
%	 We will define  an e.i.\  $\SI 3 $   subalgebra $\+ B$  of   the complemented elements in $[D,A]/=^*$, for an appropriate r.e.\ set $D \sub A$.

	The following  technical tool parallels Lemma~\ref{lem:pair_Sigma_2}.
	\begin{lemma}\label{lem:pair_Sigma_3} Given a disjoint pair of $\SI 3$ sets $U, V$, we can effectively in $x$  determine an r.e.\  splitting  $A = E_x \sqcup F_x$ such that \bc $x\in U \to E_x \text{ is computable}$, and $x\in V \to F_x \text{ is computable}$. \ec
	\end{lemma}
	
	\begin{proof}[Proof of the lemma]  By \cite[pg.\ 66]{Soare:87} we may choose   a uniformly r.e.\ double sequence  $(P_{x,n})\sN{x,n} $ of initial segments of $\omega$
		 such that
		\bc $x\in U \lra \ex n \, [P_{x,2n} = \omega]$ and $x\in V \lra  \ex n \, [P_{x,2n+1} = \omega]$.  \ec
	We fix $x$. Let $g(k,s)$ be the greatest $t\le s$ such that $t=0$, or  $P_{x,k,t-1} \neq P_{x,k,t}$.   We  enumerate sets $E_x, F_x$ with    $E_x \sqcup  F_x= A$ as follows.
		At stage  $s> 0$, if  $a \in A_s - A_{s-1}$, see if there is a  $k \le s$ be least such that $a < g(k,s)$. If so, choose $k$ least. If $k$ is even, enumerate $a$ into $F_x$; otherwise, or if there is no such $k$, enumerate $a $ into~$E_x$.
		
		First suppose that $x \in U$. Then   the  least $k$ such that  $P_{x,k} = \omega$ is even. We have $\lim_s g(k,s) =  \infty$. If $k>0$ then  also   $r= \max_{i<k} \lim_s g(i, s) < \infty$. For any $v\ge r$,  we have $v \in E_x \lra v \in E_{x,s}$, where $s$ is least such that $g(k,s) > v$. Therefore $E_x$ is computable.
		
		If $x \in V$, then   the  least $k$ such that  $P_{x,k} = \omega$ is odd. A similar argument shows that $F_x$ is computable.
	\end{proof}
		Since $A$ is non-recursive, there is a   major subset $D \subset_m A$ (see \cite[Section X.4]{Soare:87}). Thus, if $R \sub A$ is computable then $R \sub ^* D$.
	We now define an  e.i.\   $\SI 3$ Boolean algebra  $\+ B $    which is a subalgebra  of the complemented elements in $[D,A]/=^*$. Suppose    we are given a   pair of e.i.\  $\SI 3$ sets $U,V$. Let $(E_x)\sN x$, $(F_x)\sN x$ be the u.r.e.\ sequences obtained through the foregoing lemma.
	We let $\+ B$ be the Boolean algebra generated by all the elements   of the form $(E_x \cup D)^*$ for $x \in \omega$. Since uniformly in~$x$ we can recursively  enumerate both $E_x $ and its complement $F_x = A - E_x$, it is clear that $\+B $ is a $\SI 3$ Boolean algebra. Furthermore, by Lemma~\ref{lem:pair_Sigma_3} and   Fact~\ref{fa:easy e.i.}, $\+ B$ is effectively inseparable. So  by Theorem~\ref{pro:Insep_are_complete}  the preorder of  $\+ B$, namely, $\{\la x, y \ra \colon \, E_x \sub^* E_y \cup D\}$,     is $\SI 3 $-complete.  Since the sequence $(E_x)\sN x$ is uniformly r.e., this establishes the theorem.
\end{proof}
Note that we have in fact shown that almost inclusion of sets of the form $(W_e\cup D) \cap A$ is $\SI 3$-complete.  It is worth noting  that  by  a result of Maass and Stob~\cite{Maass.Stob:83}, the   lattice $[D,A]/=^*$  is   unique up to effective isomorphism whenever $D $ is a major subset of~$A$ .
   Thus, only the complemented elements of the  Maass-Stob lattice are  needed to establish  the  completeness of $\sub ^*$ as a $\SI 3$ preorder; this lattice    is  in itself merely  a ``tiny'' part of the lattice of r.e.\ sets under almost inclusion.

\begin{cor}
	\label{cor:=* Sigma 3} The equivalence relation $\{\la e, i \ra \colon W_e =^* W_i\}$ is $\SI 3$-complete. \end{cor}
\begin{proof}
This follows immediately from Fact \ref{fact:preorder}.
\end{proof}

The relation $=^*$ on r.e.\ indices is the same relation as that called $\Ece0$
in \cite[Thm.\ 3.4]{Coskey.Hamkins.Miller:12}.
%
% \subsection{Comparison with the Borel theory}
%
In the spirit of that article, we consider this and several other
equivalence relations carried over from the Borel theory.
The relations $E_0$, $E_1$, $E_2$, and $E_3$
are all standard Borel equivalence relations, whose definitions
 can be found in~\cite{Gao:09} (and may also be gleaned from the definitions given shortly). If $A\subseteq\omega$,
we write $(A)_n=\set{x}{\la n,x\ra\in A}$ for the
$n$-th column of $A$.
The analogues of these for r.e.\ indices are relations on the natural numbers,
derived by letting each natural number $e$ represent the set $W_e$.
Thus, for an equivalence relation $E$ on reals, we let $i~E^{ce}~j$
iff $W_i~E~W_j$, yielding the following specific relations.

\begin{align*}
%\label{eq:Ece relations}
& i~\Ece0~j \qquad\iff\qquad |W_i\triangle W_j|<\infty \qquad(\iff~~~W_i=^* W_j)\\
& i~\Ece1~j \qquad\iff\qquad \forall^\infty n~\left[(W_i)_n= (W_j)_n\right]\\
& i~\Ece2~j \qquad\iff\qquad \sum_{i\in W_i\triangle W_j}\frac1{i} <\infty\\
& i~\Ece3~j \qquad\iff\qquad \forall n~\left[~|(W_i)_n\triangle (W_j)_n|<\infty\right].
\end{align*}
%Note that  $\Ece0$ coincides with $=^*$ on r.e.\ indices.
Now it is not hard to see that $\Ece3$ is $\Pi^0_4$-complete as a set (since $=^*$ is $\Sigma^0_3$-complete as a set), and therefore not computably
reducible to any of the others, since the others are all $\Sigma^0_3$.
In the Borel theory, $E_0 <_B E_1$, $E_0 <_B E_2$, and $E_0 <_B E_3$,
with no other reductions holding among these relations.  The proofs
of the Borel reducibilities can be adapted to give proofs that
$\Ece0\leq_c\Ece1$, $\Ece0\leq_c\Ece2$, and $\Ece0\leq_c\Ece3$.
Therefore, Corollary \ref{cor:=* Sigma 3} gives a further result, providing the relation left open in \cite[Fig.\ 3]{Coskey.Hamkins.Miller:12}

\begin{cor}
$\Ece0$, $\Ece1$ and $\Ece2$ are    all  $\SI 3$-complete  equivalence relations,
and hence are bireducible with each other under $\leq_c$.
\end{cor}

In contrast, Corollary \ref{cor:noPincomplete} showed
that the relation $\Ece3$ cannot be complete under $\leq_c$
at any level of the arithmetical hierarchy, as it is $\Pi^0_4$
but not $\Sigma^0_4$ by virtue of its $\Pi^0_4$-completeness as a set of pairs.

\subsection{Weak truth-table reducibility on r.e.\ sets}

It is well-known that weak truth-table reducibility $\lwtt$ on r.e.\ sets is a $\SI 3$ preorder with computable supremum operation. It   determines the    distributive  upper semilattice   of r.e.\ weak truth-table degrees; see e.g.\ \cite{Odifreddi:81}.
\begin{theorem}
	\label{thm:m} The preorder $\{\la e, i \ra \colon W_e \lwtt  W_i\}$ is $\SI 3$-complete.
\end{theorem}
\begin{proof}
Ambos-Spies~\cite{Ambos-Spies:85} called an r.e.\ set $A$ \emph{antimitotic} if for any r.e.\  splitting $A = X \sqcup Y$, the   sets $X$ and $Y$ induce  a minimal pair in the Turing degrees. He   built an incomputable antimitotic set $A$, which we will fix in the following.  Let $\mathbf a = \deg_{wtt}(A)$.

Note that the degrees $\{\deg_{wtt}(X) \colon X \sqsubseteq A \} $ form a Boolean algebra of complemented elements in $[\mathbf 0, \mathbf a]$ with the usual degree ordering. We establish an e.i.\  $\SI 3 $   subalgebra $\+ C$ of this Boolean algebra. We are given a   pair of e.i.\  $\SI 3$ sets $U,V$. Let $(E_x), (F_x)$ be uniformly r.e.\ sequences as in Lemma~\ref{lem:pair_Sigma_3}. Thus $E_x \sqsubseteq A$ for each $x$.  Let $\+ C$ be the Boolean algebra generated by the $\deg_{wtt}(E_x)$,  $x \in \omega$. Since $A$ is antimitotic, intersection, union and complementation for splits of $A$ correspond to the operations of meet, join, and complementation in $[\mathbf 0, \mathbf a]$. Since the former operations are effective, $\+ C$ is indeed a $\SI 3$ Boolean algebra.

If  $x\in U$ then $E_x$ is computable, so $\deg(E_x) = \mathbf 0$.  If  $x\in V$ then $F_x$ is computable, so $\deg(E_x) = \mathbf a$. This implies that $\+ C$ is effectively inseparable.  So the preordering of  $\+ C$, namely, $\{\la x, y \ra \colon \, E_x \lwtt  E_y\}$,    is $\SI 3 $-complete by Theorem~\ref{pro:Insep_are_complete}.  Since the sequence $(E_x)\sN x$ is uniformly r.e., this establishes the theorem.
\end{proof}

%Questions:

%How to do $\le_m$? We could use some literature result (Lachlan embeddings of certain distributive  $\SI 3$ semilattices as initial segments), or construct an  r.e.\ set A needed for the general method directly. Is it worth the trouble?

%Can we do $\le_T$?

%\input{Fragments/Andre_Sigma3}

%\input{Fragments/Russell_corollary} %Not sure where this should go Andre

%\input{Fragments/Egor_Sigma02Preorders}

\section{Turing reducibility on r.e.\ sets is a $\SI 4$-complete preorder}\label{sec:Sigmacomplete}

\newcommand{\Eperm}{\ensuremath{E^{ce}_{perm}}}
\newcommand{\Eset}{\ensuremath{E^{ce}_{set}}}
\newcommand{\comment}[1]{}
\newcommand{\leqset}{\leq_{set}}
\newcommand{\geqset}{\geq_{set}}

\newcommand{\Zce}{\ensuremath{Z^{ce}_{0}}}
\newcommand{\rs}[1]{\upharpoonright{#1}}
\newcommand{\con}{\ensuremath{*}}

\newcommand{\ais}{_{i,s}^{\alpha}}
\newcommand{\bis}{_{i,s}^{\beta}}
\newcommand{\aisp}{_{i,s+1}^{\alpha}}
\newcommand{\bisp}{_{i,s+1}^{\beta}}
\newcommand{\ajs}{_{j,s}^{\alpha}}
\newcommand{\bjs}{_{j,s}^{\beta}}
\newcommand{\ajsp}{_{j,s+1}^{\alpha}}
\newcommand{\bjsp}{_{j,s+1}^{\beta}}

\label{subsec:Turingequivalence}

\begin{thm}
\label{thm:sigma4}
The preorder $\{\la i,j \ra\mid W_i\leq_T W_j\}$ is   $\Sigma^0_4$-complete.  Hence the  equivalence  relation $\{\la i,j \ra\mid W_i\equiv_T W_j\}$ is also $\SI 4$-complete.
\end{thm}

\begin{proof}
We fix a $\Sigma^0_4$ preorder $R$. We give a    construction with the overall aim  to build a uniform sequence of r.e. sets $(V_i)\sN i$ such that $i R j$ iff $V_i\leq_T V_j$. Following Nies~\cite{Nies:97*2} we fix a uniformly r.e. sequence $X^{i,j}_{e,p}$ of initial segments of $\omega$, such that
\begin{eqnarray*}
i R j &\rightarrow&\text{For almost all $e$ and $p$, } |X^{i,j}_{e,p}|<\infty,\\
\neg i R j&\rightarrow&\forall e\exists p  \, X^{i,j}_{e,p}=\omega.
\end{eqnarray*}
Here $|X|$ denotes the cardinality of the set $X$. The facts ``$X^{i,j}_{e,p}=\omega$'' and ``$X^{i,j}_{e,p}<\infty$'' are (uniformly) $\Pi^0_2$ and $\Sigma^0_2$ respectively, so each such statement can be measured at a single node on a priority tree measuring infinitary or finitary behaviour. That is, $\neg i R j$ would be equivalent to the fact that ``$\forall e\exists p $ such that the node measuring $|X^{i,j}_{e,p}|$ has true infinitary outcome''. Similarly, if $i R j$ then ``for almost all $e,p$, the node measuring $|X^{i,j}_{e,p}|$ has true finitary outcome''. Hence the ordered pairs in $R$ will determine the true path of the construction, i.e. which nodes are visited infinitely often and which are not. We will arrange the strategies on the construction tree to align our actions with the true path.

In this proof we use $Y\rs{k}$ to mean the first $k+1$ bits of $Y$. We say ``$Y$ changes below $k$" to mean a change in $Y\rs{k}$.

\subsection{Requirements and a high-level description}

The requirements to be met are $S_{i,j}$ for $i,j\in\omega$:
\[S_{i,j} ~: ~i R j \leftrightarrow V_i\leq_T V_j. \]
The way in which each $S_{i,j}$ is satisfied can only be answered by a $0^{''''}$ oracle, and so it is difficult to split $S_{i,j}$ explicitly into subrequirements in a meaningful way. We will instead describe the local action of each node on the priority tree and then describe how the priority tree ensures that each requirement $S_{i,j}$ is met globally.

The construction tree is a labeled binary tree. Each node of length $\langle i,j,e,p\rangle$ is assigned the set $X^{i,j}_{e,p}$ and has two outcomes labelled $\infty,0$. As usual $\infty$ is to the left of $0$. If $\alpha$ is assigned the set $X^{i,j}_{e,p}$ we write \bc $i_\alpha,j_\alpha,e_\alpha,p_\alpha,X_\alpha$ for $i,j,e,p,X^{i,j}_{e,p}$, respectively.  \ec Outcome $\infty$ stands for $X_\alpha=\omega$, while outcome $0$ stands for $X_\alpha$ finite.

We now describe the local action of each node $\alpha$. During the construction $\alpha$ has a dual role. During the non-expansionary stages (i.e. $\alpha\con 0$ stages) $\alpha$ builds a Turing reduction aiming to make $V_i\leq_T V_j$. Loosely speaking, this strategy monitors for changes in $V_i$ and responds with a change in $V_j$ each time it sees an element entering $V_i$. The other (conflicting) strategy of $\alpha$ acts at each expansionary stage ($\alpha\con\infty$ stage). At each expansionary stage $\alpha$ checks to see if it can make $V_i\neq\Phi_{e_\alpha}^{V_j}$ by enumerating some number in $V_i$ and restraining $V_j$ (henceforth we will  write $\Phi_\alpha$ instead of $\Phi_{e_\alpha}$ for convenience). These two strategies are clearly conflicting, and we always initialize the Turing reduction whenever an expansionary stage occurs for $\alpha$.

We now describe how $S_{i,j}$ can be met globally. Suppose $\neg i Rj$. Then for each $e$ there is some node $\alpha$ along the true path with $i_\alpha=i,j_\alpha=j,e_\alpha=e$ and having true outcome $\infty$. This $\alpha$ will ensure that $\Phi_e^{V_j}\neq V_i$, and so overall we get $V_i\not\leq_T V_j$. In this case there could be infinitely many nodes $\beta$ along the true path also working for $i,j$, but which have true outcome $0$. The local strategy for $\beta$ will attempt to demonstrate $V_i\leq_T V_j$, which cannot possibly be achieved. Indeed each such node $\beta$ along the true path will be injured infinitely often by another node $\alpha$ further along the true path aiming to diagonalize $\Phi_e^{V_j}\neq V_i$. This is the key difference between this construction and a typical $\emptyset'''$ tree argument. In the latter case we usually only allow a finite amount of injury to each requirement along the true path, whereas in our case it is necessary that some node gets injured infinitely often along the true path.

Now suppose that $i Rj$. Then almost every node along the true path working for the pair $i,j$ must have true outcome $0$. In particular there is a shortest node $\alpha$ along the true path with true outcome $0$ (we will later call all possible candidates for $\alpha$ a \emph{top node}). In this case there could be finitely many top nodes $\alpha'\prec\alpha$ with true outcome $0$. However the Turing reducibilities built by these $\alpha'$ will all be destroyed (by nodes between $\alpha'$ and $\alpha$ which are attempting to diagonalize), and $V_i\leq_T V_j$ will be demonstrated by the final top node $\alpha$ along the true path. It is important to note that each top node $\alpha'$ cannot possibly know anything about the nodes on the true path which are longer than $\alpha'$, and so the situation described above cannot be avoided.

We conclude this section with some technical definitions. We say that $\alpha$ is a top node if there is no $\beta\prec\alpha$ such that $i_\alpha=i_\beta$ and $j_\alpha=j_\beta$, or if $\alpha(|\beta|)=\infty$ for the maximal such $\beta$. Given a node $\alpha$ we define the top of $\alpha$ to be the maximal top node $\tau\preceq\alpha$ such that
$i_\alpha=i_\tau$ and $j_\alpha=j_\tau$. We say that $\alpha$ is a child of $\tau$ if $\alpha$ has top $\tau$.

We let $Z^\infty(\alpha)=\{\beta\prec\alpha \mid \beta\con\infty\preceq\alpha$ and $\beta$ is not a top node$\}$, which are all the nodes extended by $\alpha$ running the $\infty$ strategy. Similarly we let $Z^0(\alpha)=\{\beta\prec\alpha \mid \beta\con 0\preceq\alpha$ and $\beta$ is a top node$\}$.

\subsection{Ensuring $V_i\leq_T V_j$} Each top node $\alpha$ will build a Turing reduction $\Delta_\alpha$ to ensure that $V_{i_\alpha}\leq_T V_{j_\alpha}$ in the case that $\alpha\con 0$ is along the true path. The Turing reduction is built indirectly by specifying the coding markers $\{\delta_\alpha(x)[s]\}_{x,s\in\omega}$. Each $\delta_\alpha(x)[s]$ is not in $V_{j_\alpha}[s]$ and obeys the following standard marker rules: For all $x$ and $s$, we ensure
\begin{itemize}
\item[(i)] $\delta_\alpha(x)[s]<\delta_{\alpha}(x+1)[s]$.
\item[(ii)] If $\delta_\alpha(x)[s]\neq \delta_\alpha(x)[s+1]$ or if $\delta_\alpha(x)[s]\downarrow$ and $\delta_\alpha(x)[s+1]\uparrow$ then $V_{j_\alpha,s-1}\rs{(\delta_\alpha(x)[s])}\neq V_{j_\alpha,s}\rs{(\delta_\alpha(x)[s])}$.
\item[(iii)] If $x$ is enumerated in $V_{i_\alpha}$ at stage $s$, and if $\delta_\alpha(x)[s]$ is defined then $V_{j_\alpha}$ must change below $\delta_\alpha(x)[s]$ at the same stage or later.
\item[(iv)] $\lim_s \delta_\alpha(x)[s]$ exists (if $\alpha\con 0$ is along the true path).
\end{itemize}
It is easy to see that if these rules are ensured, and if $\alpha*0$ is visited infinitely often, then $V_{i_\alpha}\leq_T V_{j_\alpha}$: To compute if $x$ is in $V_{i_\alpha}$ we find recursively in $V_{j_\alpha}$ a stage $s$ where $V_{j_\alpha}[s]$ is correct up to $\delta_\alpha(x)[s]$. This stage exists because of (iv). At all stages $t>s$ we have $\delta_\alpha(x)[s]=\delta_\alpha(x)[t]$ because of (ii). By (iii) we must have $x\in V_{i_\alpha}$ iff $x\in V_{i_\alpha}[s]$.

\subsection{Informal description of the strategy} We now describe the strategy and the method of overcoming the key difficulties. Each top node $\alpha$ will have a dual role. At $\alpha\con 0$ stages it must extend the Turing reducibility $\Delta_\alpha^{V_j}$ by picking a fresh value for $\delta_\alpha(x)$. It must also ensure the correctness of $\Delta_\alpha^{V_j}(x)$ by enumerating $\delta_\alpha(x)$ if $x$ had been enumerated into $V_i$ recently. At each $\alpha\con\infty$ stage $\alpha$ will attempt to diagonalize by enumerating (if possible) some number $l$ into $V_i$ and preserving $V_j$. Clearly these two strategies are conflicting and we give the diagonalization strategy higher priority; Hence at each $\alpha\con\infty$ stage we reset $\Delta_\alpha$ and all $\delta_\alpha$ markers. If
$\alpha\con 0$ is along the true path then we ensure $V_i=\Delta_\alpha^{V_j}$, and if $\alpha\con\infty$ is along the true path then we ensure $\Phi_\alpha^{V_j}\neq V_i$. In the latter case the reduction $\Delta_\alpha$ fails. Analyzing the overall outcome of the construction we note that if $\neg iRj$ then for every $e$ we ensure $\Phi_e^{V_j}\neq V_i$ at some node along the true path, and if $i R j$ then there is a final top node $\tau$ along the true path where $V_i=\Delta_\tau^{V_j}$ is successfully maintained. In this case there can be finitely many top nodes $\alpha$ where $\alpha\con\infty\preceq \tau$ for which $\Delta_\alpha^{V_j}$ is initialized infinitely often.

We now describe the main issue which will arise in implementing the above strategy. Suppose $\alpha\prec\beta$ are both working for the same $i,j$, and $\alpha$ has true outcome $0$, while $\beta$ has true outcome $\infty$. We may as well assume that $\alpha$ is the top of $\beta$. Now $\alpha$ and $\beta$ will have conflicting actions. At each $\beta\con\infty$ stage, we may want $\beta$ to diagonalize by putting some number $l$ in $V_i$. If $\delta_\alpha(l)$ has already been defined then $\alpha$ will later want to correct $\Delta_\alpha$ by putting $\delta_\alpha(l)$ into $V_j$, which can destroy the diagonalization previously obtained by $\beta$. We overcome this by delaying $\beta$ from diagonalizing. In the meantime, at each $\beta\con\infty$ stage we lift $\delta_\alpha(k)$ for some fixed number $k$ by enumerating $\delta_\alpha(k)$ into $V_j$ and later picking a fresh value for $\delta_\alpha$. Since there are infinitely many $\beta\con\infty$ stages, $\delta_\alpha(k)$ goes to infinity. Hence $\beta$ destroys $\Delta_\alpha$ but will now be able to meet $\Phi_\beta^{V_j}\neq V_i$ (via diagonalization or an undefined computation). $\Delta_\alpha^{V_j}$ is now destroyed. However this is fine because $\beta$ witnesses the fact that $\alpha$ is not the final top node along the true path. In this case the reduction $\Delta^{V_j}=V_i$ will be built further down the true path (if $i Rj$) or may not be built at all (if $\neg iRj$).

In the general situation (when considering nodes devoted to other pairs $i',j'$) we consider the nodes $\beta_0,\beta_1\prec\alpha$ where $\beta_0,\beta_1$ have true outcome $0$, $\alpha$ has true outcome $\infty$, $i=i_{\alpha}=i_{\beta_1}$, $j_{\beta_1}=i_{\beta_0}$ and $j_{\beta_0}=j_\alpha=j$. Again if $\alpha$ wants to diagonalize at some $\alpha\con\infty$ stage it must be careful; A number $l$ enumerated in $V_i$ might cause $\beta_1$ to later correct $\Delta_{\beta_1}^{V_{j_{\beta_1}}}$ by changing $V_{j_{\beta_1}}$, which might in turn cause $\beta_0$ to correct $\Delta_{\beta_0}^{V_j}$ by changing $V_j$, destroying the previous diagonalization attempt of $\alpha$. In this case we observe that along the true path, we must have one of the following three situations hold:
\begin{itemize}
\item[(i)] There is some child $\gamma_0$ of $\beta_0$ where $\gamma_0\con\infty$ is along the true path.
\item[(ii)] There is some child $\gamma_1$ of $\beta_1$ where $\gamma_1\con\infty$ is along the true path.
\item[(iii)] Almost every  $\gamma$ working for the same pair $\la i,j \ra$ as $\alpha$ along the true path has true $0$ outcome.
\end{itemize}
One of these three situations must hold since $R$ is transitive. If (iii) holds then it is fine for $\alpha$ to fail in diagonalizing even though it has true outcome $\infty$. If (i) or (ii) holds then $\gamma_v$ (for some $v<2$) witnesses that $\beta_v$ is not the final top node along the true path; In this case $\gamma_v$ will destroy $\Delta_{\beta_v}$ by pushing some $\delta_{\beta_v}(k)$ to infinity. In this case we delay $\alpha$ from diagonalizing until we find that $\Phi_\alpha^{V_j}(l)$ has converged on a number $l$ with use $u$ safely. That is, the combined response of $\beta_0$ and $\beta_1$ to the entry of $l$ into $V_i$ will not change $V_j$ below $u$.

\subsection{Markers for capricious destruction} Each node $\alpha$ which is not a top node is given a parameter $k_\alpha$. At each $\alpha$-expansionary stage, $\alpha$ will enumerate $\delta_\tau(k_\alpha)$ into $V_{j_\tau}$ and lift this marker. Hence if $\alpha\con \infty$ is along the true path then $\delta_\tau(k_\alpha)$ goes to infinity. This is fine because $\tau$ is not the maximal top node along the true path; in this case there may be a top node further down the true path, or there may be none along the true path, and so the Turing reducibility witnessing $V_{i_\tau}\leq_T V_{j_\tau}$ need not be built at $\tau$. The important thing here is to ensure that for each top node $\tau$ where $\tau\con 0$ is along the true path and $\tau$ is not the maximal top node along the true path, we have to destroy the $\tau$ markers and create an interval $(i,\infty)$ where every point is eventually cleared of $\tau$ markers. This allows a conflicting diagonalization strategy of lower priority to succeed.

If $\tau$ is the final top node along the true path then each $\tau$ marker is lifted finitely often this way, and so the Turing reducibility witnessing $V_{i_\tau}\leq_T V_{j_\tau}$ will be built at $\tau$.

\subsection{Believable computations} A computation $\Phi_\alpha^{V_{j_\alpha}}(l)[s]$ with use $u$ is said to be $\alpha$-believable at stage $s$ if the following two conditions are met.
\begin{itemize}
\item For every $\beta\in Z^0(\alpha)$ there is no $\beta$ marker below $u$ which is pending. A $\beta$ marker $\delta_\beta(x)$ is said to be pending at $t$ if $x$ was enumerated into $V_{i_\beta}$ at some $s<t$ where $\delta_\beta(x)[s]$ is defined and $V_{j_\beta}$ has not changed below $\delta_\beta(x)[s]$.
\item There does not exist a sequence of distinct nodes $\beta_0,\cdots,\beta_n\in Z^0(\alpha)$ such that $i_{\beta_n}=i_\alpha$, $j_{\beta_0}=j_\alpha$, and for every $m<n$, $j_{\beta_{m+1}}=i_{\beta_{m}}$, and $\delta_{\beta_m}(l)[s]\downarrow\leq u$.
\end{itemize}
The first item says that a computation is $\alpha$-believable if there is no $\delta$ marker below the use that will be enumerated by some $\beta\prec\alpha$ due to coding. The second item ensures that if $l$ were to be enumerated by $\alpha$ into $V_{i_\alpha}$ for the sake of diagonalization then the resulting sequence of correction actions will not cause the computation to be later destroyed.

\subsection{Formal construction} At stage $0$ we initialize every node. This means to reset $\delta_\alpha$ and $k_\alpha$ to be undefined. At stage $s>0$ we define $TP_s$ of length $s$, where $TP_s$ is the stage $s$ approximation to the true path of construction. Assume that $\alpha\prec TP_s$ has been defined. If $|X_\alpha|$ has increased since the last $\alpha$-stage then we play outcome $\infty$ for $\alpha$, otherwise play outcome $0$ for $\alpha$. Initialize every node to the right of $\alpha\con o$ where $o$ is the outcome played. We now act for $\alpha$. There are four cases.
\begin{itemize}
\item[(i)] $\alpha$ is not a top node and outcome $0$ is played. Do nothing.
\item[(ii)] $\alpha$ is a top node and outcome $\infty$ is played. Initialize $\alpha$. Check if it is possible to diagonalize, and if so, do it. This means to check if there exists some $l<s$ such that  $\Phi_\alpha^{V_{j_\alpha}}(l)[s]\downarrow$ is $\alpha$-believable, $l\in \omega^{[\alpha]}-V_{i_\alpha}[s]$, and for every $l'\leq l$ we have $V_{i_\alpha}(l')=\Phi_\alpha^{V_{j_\alpha}}(l')[s]$. We also require $l$ to be larger than the previous stage where $\alpha$ was initialized by another node. If such $l$ is found enumerate $l$ into $V_{i_\alpha}$ and initialize every node extending $\alpha$.
\item[(iii)] $\alpha$ is not a top node and outcome $\infty$ is played. Let $\tau$ be the top of $\alpha$. If $k_\alpha\uparrow$ we pick a fresh value for it. Otherwise if $\delta_\tau(k_\alpha)\downarrow$ enumerate it into $V_{j_\tau}$ and make it (and all larger $\tau$ marker) undefined. Check if it is possible to diagonalize, and if so, do it.
\item[(iv)] $\alpha$ is a top node and outcome $0$ is played. We will correct and extend the Turing reduction $\Delta_\alpha$. For correction, we check to see if there is any number $l$ enumerated into $V_{i_\alpha}$ since the last visit to $\alpha$. Let $l$ be the least. If $\delta_\alpha(l)\downarrow$ we enumerate it in $V_{j_\alpha}$ and make $\delta_\alpha(l)\uparrow$. To extend $\Delta_\alpha$ we let $k<s$ be the least such that $\delta_\alpha(k)\uparrow$. We pick a fresh value $x\in\omega^{[\alpha]}$ and set $\delta_\alpha(k)\downarrow=x$.
\end{itemize}

\subsection{Verification} Let $TP=\lim\inf_s TP_s$ be the true path of the construction, which clearly exists as the construction tree is finitely branching. We first prove the following key lemma.

\begin{lemma}\label{lem:pi4key}Suppose $\alpha$ is a node on the true path. Then $\alpha$ makes finitely many diagonalization attempts.
\end{lemma}
\begin{pf}
Assume we are at a stage of the construction where each $\beta\prec\alpha$ makes no more diagonalization attempt, and we never move left of $\alpha$. Suppose $\alpha$ makes infinitely many diagonalization attempts after this. Let $l$ be the smallest number put in by $\alpha$ due to a diagonalization attempt, at some stage $t$. At $t$ we have $\Phi_\alpha^{V_{j_\alpha}}(l)[t]\downarrow=0$ with use $u$, and is an $\alpha$-believable computation. We claim that at $t$, $V_{j_\alpha}$ is stable below $u$, which contradicts the assumption that infinitely many attempts are made.

By convention $u<t$, and so the only nodes which can enumerate a number less than $u$ into any set after stage $t$, are nodes of the form $\beta\preceq\alpha$. In fact we cannot have $\beta=\alpha$ because if $\alpha$ is a top node then it is initialized at $t$ and any further diagonalization action by $\alpha$ involves a number $l'>l$, and if $\alpha$ is not a top node then $k_\alpha$ or $\delta_\tau(k_\alpha)$ is undefined before the diagonalization at $t$ (here $\tau$ is the top of $\alpha$). Similarly we cannot have $\beta\con\infty\preceq\alpha$, because if $\beta$ is a top node then it is initialized at $t$ and never again performs diagonalization, and if $\beta$ is not a top node then $k_\beta$ or $\delta_\tau(k_\beta)$ is undefined when $\alpha$ was visited at $t$. Hence if $\beta$ is a node which enumerates a number smaller than $u$ into any set after $t$, it must satisfy $\beta_0\in Z^0(\alpha)$.

Let $\beta_0\in Z^0(\alpha)$ be the first node after $t$ to change $V_{j_\alpha}$ below $u$, say at $t_0>t$. At $t_0$, $\beta_0$ is correcting $\Delta_{\beta_0}$ and responding to the action of $\beta_1$ at $t_1$, $t_1<t_0$. Let $l_0$ be the number which was enumerated by $\beta_1$ into $V_{i_{\beta_0}}$ at $t_1$. We cannot have $t_1<t$ because of $\alpha$-believability at $t$. Suppose $t_1>t$. Since $l<\delta(l)$ for every $l,\delta$, we have $\beta_1\in Z^0(\alpha)$. Then we have correction at $t_1$ and so $j_{\beta_1}=i_{\beta_0}$. We can then repeat to get a sequence $\beta_0,\cdots,\beta_{n}\in Z^0(\alpha)$ and $t_0>\cdots>t_n>t_{n+1}=t$, where
$\beta_{n+1}=\alpha$, $j_{\beta_0}=j_\alpha$, $i_{\beta_n}=i_\alpha$, and for every $m<n$, $j_{\beta_{m+1}}=i_{\beta_{m}}$. Furthermore for every $m\leq n$, $l_{m}<\delta_{\beta_{m}}(l_{m})=l_{m-1}$, and $l_n=l$. Since $\delta_{\beta_m}(l_m)[t_m]\downarrow<u$ we have $\delta_{\beta_m}(l_m)[t]\downarrow<u$ (since $\delta$ markers are picked fresh), and so $\delta_{\beta_m}(l)[t]\downarrow< u$. By removing any cycles in the sequence, we may assume that the nodes $\beta_0,\cdots,\beta_n$ are distinct. This contradicts the second requirement for $\alpha$-believability at $t$.
\qed
\end{pf}

We devote the remainder of the proof to showing that each requirement $S_{i,j}$ is met. Fix $i$ and $j$ such that $iRj$. Let $\alpha$ be the shortest node along $TP$ such that $i_\alpha=i$, $j_\alpha=j$ and for every node $\beta$ such that $i_\beta=i$ and $j_\beta=j$ where $\alpha\preceq\beta\prec TP$, we have $\beta\con 0\preceq TP$. Clearly $\alpha$ is a top node. We argue that $V_i=\Delta_\alpha^{V_j}$ (by a symmetric argument we can then conclude $V_i\equiv_T V_j$). Since $\alpha$ has true $0$ outcome there are finitely many self-initializations, and by Lemma \ref{lem:pi4key}, there are only finitely many initializations to $\alpha$. The marker rules (i), (ii) and (iii) are clearly met for $\delta_\alpha$. We argue (iv) holds. Since $\alpha$ is visited infinitely often it suffices to check that each marker is made undefined finitely often. Each $\delta_\alpha(k)$ can only be made undefined by some child of $\alpha$, and only when $k=k_\alpha$. A child of $\alpha$ to the left of the true path only does this finitely often, while a child to the right of the true path has $k_\alpha=k$ at finitely many stages. A child on the true path must have true $0$ outcome and acts at only finitely many stages.

Now fix $i$ and $j$ such that $\neg i R j$. Fix an $e$, and let $\alpha$ be the node on the true path such that $i_\alpha=i,j_\alpha=j,e_\alpha=e$ and $\alpha\con\infty\prec TP$. We argue that $\Phi_\alpha^{V_j}\neq V_i$. Suppose that $\Phi_\alpha^{V_j}= V_i$.

\begin{lemma}\label{lem:pi42}For almost every $l$, the true computation $\Phi_\alpha^{V_j}(l)$ is eventually $\alpha$-believable.
\end{lemma}
\begin{pf}Call a sequence of distinct nodes $\beta_0,\cdots,\beta_n\in Z^0(\alpha)$ such that $i_{\beta_n}=i_\alpha$, $j_{\beta_0}=j_\alpha$, and for every $m<n$, $j_{\beta_{m+1}}=i_{\beta_{m}}$, a bad sequence. There are only finitely many bad sequences, and for each bad sequence there exists a node $\gamma$ which is a child of some $\beta_m$ such that $\gamma\con\infty\prec TP$. This follows by the transitivity of $R$. Since $\gamma$ is not a top node it is initialized finitely often and so achieves a stable value for $k_\gamma$. We consider $l$ larger than every $k_\gamma$ and argue that $\Phi_\alpha^{V_j}(l)$ is eventually $\alpha$-believable.

Fix $l$ and fix a large stage $s$ such that $\alpha$ is visited and $\Phi_\alpha^{V_j}(l)[s]\downarrow$ with the correct use $u$. The first requirement for $\alpha$-believability must be met because any pending marker from any node in $Z^0(\alpha)$ must be enumerated in $V_{j}$ below $u$ before the next visit to $\alpha$. For the second requirement, we can assume that $s$ is large enough such that for each bad sequence and associated $\gamma$, we have $\delta_{\beta_m}(k_\gamma)[s]>u$. This is possible because at each visit to $\gamma\con\infty$ we make $\delta_{\beta_m}(k_\gamma)$ undefined.

Let us assume the second requirement for $\alpha$-believability fails. Fix a bad sequence witnessing this, and let $\gamma$ be the associated node. By the failure of believability we have $\delta_{\beta_m}(l)[s]\leq u$; On the other hand by the choice of $s$ we have $\delta_{\beta_m}(l)[s]>\delta_{\beta_m}(k_\gamma)[s]>u$, a contradiction.
\qed
\end{pf}

By Lemma \ref{lem:pi4key}, there are only finitely many initializations to $\alpha$ initiated by a different node; Let $s$ be large enough so that there are no more initializations to $\alpha$ of this type. By Lemma \ref{lem:pi42} pick $l$ large enough, and wait for $\Phi_\alpha^{V_j}(l)$ to become $\alpha$-believable, and for $\Phi^{V_j}_\alpha$ to agree with $V_i$ below $l$. Hence $\alpha$ will make another diagonalization attempt after $s$. Since this holds for any large $s$, we have a contradiction to Lemma  \ref{lem:pi4key}. This ends the proof of Theorem \ref{thm:sigma4}.
\end{proof}

We also obtain as a corollary, examples of complete equivalence relations at the $\Sigma^0_{n+4}$ level for each $n\in\omega$:

\begin{cor}
\label{cor:Turing}
For each $n\in\omega$, the preorder $\{\la i,j \ra\mid W_i^{\emptyset^{(n)}}\leq_T W_j^{\emptyset^{(n)}}\}$ is  $\SI {n+4}$-complete.
\end{cor}
\begin{proof}
The proof of Theorem \ref{thm:sigma4} produces a computable function $f$ such that for each $i,j$, we have $iRj$ iff $W_{f(i)}\leq_T W_{f(j)}$. Relativizing this to $\emptyset^{(n)}$, for each $\Sigma^0_{n+4}$ preorder $R$, we get a function $f\leq \emptyset^{(n)}$ such that $iRj$ iff $W^{\emptyset^{(n)}}_{f(i)}\leq_T W^{\emptyset^{(n)}}_{f(j)}\oplus \emptyset^{(n)}$.

For each oracle $X$ and each function $g\leq_T X$ there is a computable function $h$ such that $W^X_{g(i)}=W^X_{h(i)}$ for every $i$. To see this, observe that the set $\{(n,i):n\in W^X_{g(i)}\}$ is $\Sigma^{0,X}_1$ and is hence equal to the domain of $\Psi^X$ for some functional $\Psi$. The function $h$ is obtained by applying the s-m-n Theorem.

Now applying this fact we may assume that $f$ is computable. Let $\hat{f}$ be a computable function such that $W^X_{f(i)}\oplus X=W^X_{\hat{f}(i)}$ for all $X$ and $i$. Then $\hat{f}$ witnesses that $R$ is reducible to the preorder $\{\la i,j \ra\mid W_i^{\emptyset^{(n)}}\leq_T W_j^{\emptyset^{(n)}}\}$.
\end{proof}

\begin{cor}
For each $n\in\omega$, the preorder $\{\la i,j \ra\mid W_i^{(n)}\leq_T W_j^{(n)}\}$ is  $\Sigma^0_{n+4}$-complete.
\end{cor}
\begin{proof} The proof of Corollary \ref{cor:Turing} shows that the preorder $\{\la i,j \ra\mid W_i^{\emptyset^{(n)}}\oplus\emptyset^{(n)}\leq_T W_j^{\emptyset^{(n)}}\oplus\emptyset^{(n)}\}$ is complete among the $\Sigma^0_{n+4}$ preorders. By the uniformity of Sacks' Jump Inversion Theorem (see \cite[Corollary VIII.3.6.]{Soare:87}), for each $n$ there is a computable function $q_n$ such that for all $x$, we have $W_{q_n(x)}^{(n)}\equiv_T W_x^{\emptyset^{(n)}}\oplus\emptyset^{(n)}$.
\end{proof}

\begin{cor}For each $n\in\omega$ the relation $$ \{\la i,j \ra: W_i^{\emptyset^{(n+1)}}\equiv_1 W_j^{\emptyset^{(n+1)}}\}$$
is $\Sigma^0_{n+4}$-complete.\end{cor}

\section{Discussion and open questions}

Many   effective equivalence relations from algebra can be considered under the aspect of relative complexity. For instance:
\begin{question}
	Is the equivalence relation  $E$ of isomorphism between  finite presentations of  groups $\SI 1$-complete?
\end{question}
Adyan and Rabin    independently showed  in 1958 that  the triviality problem, whether a finite presentation describes the trivial group, is $m$-complete. See  Lyndon and Schupp~\cite[IV.4.1]{Lyndon.Schupp:77}, letting the given group $H$ there have a  word problem $m$-equivalent to the halting problem. In fact, by the discussion following the proof there, every single equivalence class of $E$ is $m$-complete: Being isomorphic to a particular finitely presented group $P$  is incompatible with free products in their sense because $P*A$ has higher rank than $P$ for any nontrivial group~$A$.

\begin{question}
	Is the equivalence relation  $E$ of isomorphism between automatic equivalence relations  $\PI 1$-complete?
\end{question}
  Kuske, Liu and Lohrey \cite{MR2953905} showed that this    is $\PI 1 $-complete as a set of pairs.  A similar question can be asked about the $\PI 1 $ relation of elementary equivalence of automatic structures for the same finite signature.

    Computable isomorphism of  {computable}   Boolean algebras  is $\SI 3$ complete by~\cite{Fokina.Friedman.ea:12}. By \cite{Fokina.Friedman.Harizanov.Knight.McCoy.Montalban:10}, isomorphism on  many natural classes of computable structures, such as graphs, is  a  $\Sigma^1_1$-complete equivalence relation. 
The third author has  observed that isomorphism of computable Boolean algebras also is  a $\Sigma^1_1$-complete equivalence relation. This uses the fact  that the coding of graphs into countable Boolean algebras from \cite{Camerlo.Gao:01}  is effective. Thus, uniformly in a computable graph it produces a computable Boolean algebra.

%\bibliographystyle{plain}
%\bibliography{bibs/eqrels,bibs/andre,bibs/egor,bibs/russell}

%

\end{document}